\documentclass{amsart}
\usepackage[margin=1.5 in]{geometry} 
\usepackage{amssymb, amsmath, amsfonts, amsthm}
\usepackage{bbm}
\usepackage{hyperref}
\usepackage{cleveref}
\usepackage{mathrsfs}
\usepackage{graphicx} 

\numberwithin{equation}{section}

    {\left[ \begin{smallmatrix}
    }
    { 
     \end{smallmatrix} \right]
    }

\newcommand{\newword}[1]{\textbf{#1}} 

\newcommand{\capstab}{\cap_{\mathrm{stab}}}

\DeclareMathOperator\codim{codim}
\DeclareMathOperator\Hull{conv}
\DeclareMathOperator\Span{Span}
\newcommand\Uniform[2]{\mathrm{U}_{#1,#2}}
\DeclareMathOperator\Sym{Sym}
\DeclareMathOperator\PPoly{PP}
\DeclareMathOperator\MW{MW}
\DeclareMathOperator\rank{rank}
\DeclareMathOperator\corank{corank}
\DeclareMathOperator\Crem{Crem}
\DeclareMathOperator\Berg{Berg}
\DeclareMathOperator\Cone{Cone}

\DeclareMathOperator\cl{cl} 
\DeclareMathOperator\gr{gr}
\DeclareMathOperator\flip{flip} 
\renewcommand\path{\operatorname{path}}
\newcommand\verts{\operatorname{verts}}
\DeclareMathOperator\Stress{crowd} 
\let\stress\Stress
\DeclareMathOperator\StressHull{CrowdHull}
\DeclareMathOperator\IndicatorGroup{\mathbb I}
\newcommand\Tangent{\mathcal T} 
\DeclareMathOperator\Chains{Chains}
\DeclareMathOperator\StressedSets{\mathcal{S}^{\textnormal{crowd}}} 
\DeclareMathOperator\StressedFlats{\mathcal{F}^{\textnormal{crowd}}}
\newcommand\halfopenpolytope{\mathring{\Delta}} 

\newcommand{\CC}{\mathbb{C}}

\newcommand{\RR}{\mathbb{R}}

\newcommand{\ZZ}{\mathbb{Z}}

\newcommand{\cB}{\mathcal{B}}

\renewcommand{\tt}{\texttt{t}}
\newcommand{\tu}{\texttt{u}}
\newcommand{\tv}{\texttt{v}}

\newcommand{\be}{\boldsymbol{e}} 
\newcommand{\bOne}{\mathbbm{1}}

\theoremstyle{plain}
\newtheorem{thm}{Theorem}[section]

\newtheorem{theorem}[thm]{Theorem}

\newtheorem{prop}[thm]{Proposition}

\newtheorem{proposition}[thm]{Proposition}

\newtheorem{lemma}[thm]{Lemma}

\newtheorem{cor}[thm]{Corollary}

\newtheorem{conj}[thm]{Conjecture}

\newtheorem*{prop*}{Proposition}
\newtheorem*{thrm*}{Theorem}

\theoremstyle{definition}

\newtheorem{example}[thm]{Example}
\newtheorem{eg}[thm]{Example}

\newtheorem{remark}[thm]{Remark}

\newtheorem{definition}[thm]{Definition}

\usepackage{color}

\newcommand{\margincolor}{red}      
\definecolor{darkgreen}{rgb}{0,0.7,0}

\newcounter{margincounter}
\newcommand{\marginnum}{
\ifnum\value{margincounter}<10
\textcolor{\margincolor}{\begin{picture}(0,0)\put(2.2,2.4){\circle{9}}\end{picture}\footnotesize\arabic{margincounter}}
\else\ifnum\value{margincounter}<100
\textcolor{\margincolor}{\begin{picture}(0,0)\put(4.256,2.5){\circle{11}}\end{picture}\footnotesize\arabic{margincounter}}
\else
\textcolor{\margincolor}{\begin{picture}(0,0)\put(6.8,2.5){\circle{14}}\end{picture}\footnotesize\arabic{margincounter}}
\fi\fi
}

\title{The omega invariant of a matroid}
\author{Alex Fink, Kris Shaw and David E Speyer}
\subjclass{05B35 (14T15 52B45)}

\begin{document}

\maketitle
\begin{abstract}
The third author introduced the $g$-polynomial $g_M(t)$ of a matroid, a covaluative matroid statistic which is unchanged under series and parallel extension.
The $g$-polynomial of a rank $r$ matroid $M$ has the form $g_1 t + g_2 t^2 + \cdots + g_r t^r$. 
The coefficient $g_1$ is Crapo's classical $\beta$-invariant. 
In this paper, we study the coefficient $g_r$, which we term the $\omega$-invariant of $M$.
We show that, if $M/F$ is connected for every proper flat $F$ of $M$, and $\omega(N)$ is nonnegative for every minor $N$ of $M$, then all the coefficients of $g_M(t)$ are nonnegative. 
We give several simplified versions of Ferroni's formula for $\omega(M)$, and compute $\omega(M)$ when $r$ or $|E(M)|-2r$ is small.
\end{abstract}
%

\section{Introduction}

\subsection{Background}
Let $E$ be a finite set with $n$ elements. 
Let $\RR^E$ be the vector space of real valued functions on $E$.
For $S \subseteq E$, let $e_S$ be the characteristic function of $S$; we abbreviate $e_{\{ i \}}$ to $e_i$.

The \newword{matroid base polytope} of a matroid $M$, denoted $\Delta(M)$, is the convex hull of $\{ e_B : B \in \cB \}$,
where $\cB$ is the set of bases of~$M$.
In terms of the rank function of~$M$ \cite[(9), (23)]{Edmonds70},
\begin{equation}\label{HDescription}
\Delta(M)=\{x\in\mathbb\RR^E : \sum_{i\in S}x_i\le\rank S 
\text{ for all }S\subseteq E,\text{ with equality for }S=E\}.
\end{equation}
The dimension of $\Delta(M)$ is $n-c(M)$, where $c(M)$ is the number of connected components of $M$.
One presentation of the basis axioms for a matroid 
is that a nonempty collection $\cB$ of $r$-element subsets of $E$ is the set of bases of a matroid if and only if every edge of the polytope $\Hull(e_B : B \in \cB)$
is parallel to $e_i - e_j$ for some $i$, $j \in E$~\cite[Thm~4.1]{GGMS}.

In the last twenty years, there has been an increased interest in matroidal decompositions and matroid valuations:
see for instance the surveys \cite{KatzMatroidTheory,ArdilaGeometries}. 
A \newword{matroidal decomposition} is a decomposition of a matroid base polytope into pieces, each of which is itself a matroid base polytope.
Matroidal decompositions are of interest for example because of connections to 
Hacking, Keel and Tevelev's compactification of the moduli space of hyperplane arrangements~\cite{HackingKeelTevelev},
linear spaces in tropical geometry \cite[\S4.4]{MaclaganSturmfels},
and gross substitutes in economics \cite{GrossSubstitutes}.

A function of matroids is valuative if it is linear in a suitable sense on matroidal decompositions.
Given a decomposition of a matroid polytope $\Delta(N)$, viewed as a polyhedral complex,
say that its \newword{interior} faces are those not contained in the boundary of $\Delta(N)$.
Then a function $v$ taking values in some abelian group is \newword{valuative} such that, if the interior faces of a matroidal decomposition $\Delta(N)$ are the polytopes $\Delta(M_j)$, then we have
\begin{equation}
 v(N) = \sum_j (-1)^{\dim \Delta(N) - \dim \Delta(M_j)}\, v(M_j) . \label{ValuativeEquation}
 \end{equation}
We can remove the signs from eq.~\ref{ValuativeEquation} by defining $v^\circ(M) = (-1)^{c(M)-1} v(M)$; then we have
\begin{equation}
 v^\circ(N) = \sum_j v^\circ(M_j) . \label{ValuativeEquationTwisted}
 \end{equation}
Following \cite{DerksenFink}, we will call a matroid function $v^\circ$ \newword{covaluative} if it obeys eq.~\ref{ValuativeEquationTwisted}.
 
Let $\Uniform{r}{n}$ be the \newword{uniform matroid of rank $r$ on $[n]$}.
Based on many computations with tropical linear spaces, the third author made the following conjecture, which is still open:
\begin{conj} \label{fVectorConj}
In any matroidal decomposition of $\Delta(\Uniform{r}{n})$, the number of faces of dimension $n-i$ is at most $\tfrac{(n-i-1)!}{(r-i)!(n-r-i)!(i-1)!}$. Equality occurs if and only if each polytope in the decomposition comes from a direct sum of series-parallel matroids.
\end{conj}

With the goal of proving Conjecture~\ref{fVectorConj}, the third author \cite{MatroidK} (see also~\cite{KTutte}) introduced a polynomial $g_M(t)\in\ZZ[t]$ which is a covaluative matroid invariant, such that 
$g_M(t) = 0$ if $M$ contains a loop or coloop,
$g_M(t) = t^i$ if $M$ is a direct sum of $i$ series-parallel matroids (i.e.\ series-parallel extensions of $\Uniform{1}{2}$), and 
$g_{\Uniform{r}{n}}(t) = \sum_{i=1}^{\min(r, n-r)} \tfrac{(n-i-1)!}{(r-i)!(n-r-i)!(i-1)!} t^i$. 
As explained in \cite{MatroidK} (shortly after \cite[Proposition~3.3]{MatroidK}),
Conjecture~\ref{fVectorConj} would follow if we could prove:
\begin{conj} \label{gPosConj}
For any matroid $M$, the coefficients of $g_M(t)$ are nonnegative. 
\end{conj}
The third author proved Conjecture~\ref{gPosConj} for matroids realizable over $\CC$~\cite{MatroidK}, and thus proved Conjecture~\ref{fVectorConj} when all matroids in the decomposition are realizable over $\CC$.

The polynomial $g_M(t)$ has constant term $0$ unless $E=\emptyset$, and has degree $\leq \min(r, n-r)$. 
Thus, excluding the $E=\emptyset$ case, $g_M(t)$ is of the form $\sum_{i\geq 1} g_i(M) t^i$ for some matroid invariants $g_i(M) \in \ZZ$, where $g_i(M)=0$ for $i > \min(r, n-r)$.
The first coefficient $g_1(M)$ is a classical matroid invariant called $\beta(M)$ \cite[Theorem 5.1]{MatroidK}, \cite[p.~2701]{KTutte}, due to Crapo~\cite{Crapo}.
The aim of this paper is to study the coefficient at the other end of the sum. Specifically, we define 
\[ \omega(M) = g_r(M)  \]
when $M$ is a matroid of rank~$r$. So $\omega(M)=0$ if $r>n-r$.

We record the behavior of $g$ and $\omega$ under direct sums.
\begin{prop} \label{multiplicative}
Let $M = M_1 \oplus M_2$. Then $g_M(t) = g_{M_1}(t) g_{M_2}(t)$ and $\omega(M) = \omega(M_1) \omega(M_2)$.
\end{prop}

\begin{proof}
The first result is \cite[Proposition 7.2]{KTutte}. 
For the second, let $r_i$ be the rank of $M_i$.  So $g_{M_i}(t)$ is a polynomial of degree $\leq r_i$, and $\omega(M_i)$ is the coefficient of $t^{r_i}$ in $g_{M_i}(t)$; similarly, $\omega(M)$ is the coefficient of $t^{r_1+r_2}$ in $g_M(t) = g_{M_1}(t) g_{M_2}(t)$. So $\omega$ is multiplicative because the leading coefficients of polynomials multiply.
\end{proof}

The name ``omega" for our invariant is meant to remind the reader of its position among the coefficients of $g$: The constant term of $g$ is zero, the linear term of $g$ is $\beta$ and the last term of $g$ (if $r \leq n-r$) is $\omega$, corresponding to the positions of $\beta$ and $\omega$ in the Greek alphabet.
By a happy coincidence, Matt Larson 
has discovered a relationship between $\omega$ and the canonical bundle of the wonderful variety%
\footnote{This was mentioned at the BIRS Matroid Theory week but is not to our knowledge published.
Compare \cite[Section 5.1]{EurLarson}.}; 
the canonical bundle to an algebraic variety $X$ is denoted $\omega_X$.

Of course, we conjecture the specific case of Conjecture~\ref{gPosConj} which applies to $\omega$:
\begin{conj} \label{omegaPosConj}
For any matroid $M$, we have $\omega(M) \geq 0$.
\end{conj}

\subsection{Main results}
The first of our main results is to identify positivity of $\omega$ as a key case for resolving positivity of $g$:

\begin{theorem} \label{OmegaPosReduction}
Let $M$ be a matroid. Suppose that we knew $\omega(N) \geq 0$ for every minor $N$ of $M$, and suppose also that $M/F$ is connected for all proper flats $F$ of $M$. Then all coefficients of $g_M(t)$ are nonnegative. 
\end{theorem}
We prove Theorem~\ref{OmegaPosReduction} in Section~\ref{ProofOPR}.

In an earlier version of this paper, we did not include the condition that $M/F$ is connected in Theorem~\ref{OmegaPosReduction}. 
We believe that this condition is not necessary, because we believe Conjecture~\ref{gPosConj}, but we do not know how to proceed without this condition. 
We thank Matt Larson for bringing this issue to our attention. 
See further discussion in Example~\ref{DisconnectedExample}.

The $\omega$ invariant vanishes whenever $r > n-r$ or, in other words, $n > 2r$.  Our next results are a computation of $\omega$ when $r \leq 4$. 
The invariant $\omega$ is unaltered by parallel extension, so we may assume without loss of generality that $M$ does not contain parallel elements. Also, $\omega(M)$ is $0$ if $M$ has loops, so we may assume that it does not.

\begin{theorem} \label{LowRankSummary}
Let $M$ have no parallel elements, and no loops. 
If $r=2$, then with these assumptions $\omega(M) = n-3$. If $r=3$, then let $\ell_1$, $\ell_2$, \dots, $\ell_k$ be the cardinalities of the rank $2$ flats of $M$; we have $\omega(M) = \binom{n-4}{2} - \sum \binom{\ell_i-2}{2}$. 
\end{theorem}

We will prove these formulae and another for the $r=4$ case in Section~\ref{LowRankSection},
where we also prove that $\omega(M) \geq 0$ for $M$ of rank~$3$,
which is not obvious from Theorem~\ref{LowRankSummary}.

At the other extreme, we can also give formulas for $\omega(M)$ for $n$ close to $2r$.

\begin{theorem} \label{MiddleRankSummary}
If $n=2r$, then $\omega(M)$ is either $0$ or $1$, and we will give a criterion for when each occurs in Section~\ref{NearMiddleMatroidSection1}. If $n=2r+1$, then $0 \leq \omega(M) \leq r$, and we will give a formula for $\omega(M)$ in Section~\ref{NearMiddleMatroidSection2}.
\end{theorem}

The proofs of Theorems~\ref{LowRankSummary} and~\ref{MiddleRankSummary} rely on formulas for expressing $\omega(M)$ as an alternating sum of $\omega$ evaluated at Schubert matroids. Derksen and Fink~\cite{DerksenFink} provided a formula to write an arbitrary matroid polytope as an alternating sum of Schubert matroid polytopes; we will need variants of this formula, which we prove in Section~\ref{EightFormulaeSection}. 
We will then use a formula for $\omega$ of a Schubert matroid due to Ferroni~\cite{Ferroni}.

\subsection*{Acknowledgments} 
The authors began the collaboration which led to the current work at the Fields Institute semester on combinatorial algebraic geometry in 2016, and renewed their connection at the Banff International Research Station's week on Matroid Theory in 2023; we appreciate the excellent working conditions at both of these research centers.

The authors are grateful to the authors of~\cite{BEST} for sharing early drafts of their work, and for writing clear, straightforward proofs that have replaced many of our longer arguments;
conversations with Andrew Berget have been especially valuable.
We also thank Matt Larson for many useful conversations, for encouraging us to complete this paper, and particularly for raising the importance of the connectedness condition on $M/F$.

During various points in this research,
the first author was supported by Engineering and Physical Sciences Research Council grants EP/M01245X/1 and EP/X001229/1
and by the European Union's Horizon 2020 research and innovation programme under the Marie Sk\l odowska-Curie grant agreement No.~792432. The second author's research was supported by a Jerrald E. Marsden postdoctoral fellowship at the  Fields Institute, a postdoctoral research fellowship from the Alexander von Humboldt foundation, and 
the TMS foundation project ``Algebraic and topological cycles in complex and tropical geometries".
The third author was supported by DMS-1600223, DMS-1855135, DMS-1854225 and DMS-2246570.

\addtocontents{toc}{\smallskip \textbf{Part 1: Reduction of $g$-positivity to $\omega$-positivity}}

\section{Piecewise polynomials and Minkowski weights}\label{PiecewisePolynomialsSection}
Let $V$ be a finite-dimensional real vector space, and $V^\vee$ its dual. We assume that $V$ is equipped with an integer structure, meaning a discrete full rank lattice $V_{\ZZ}$ inside $V$. 
Let $\Sigma$ be a complete fan in $V$.  We will also assume that $\Sigma$ is rational, meaning that each cone is generated by finitely many vectors in~$V_{\ZZ}$. 
We define the \newword{lineality space} $L_{\Sigma}$ of $\Sigma$ to be the minimal face of $\Sigma$, which is also the maximal subspace of $V$ such that $\Sigma$ is invariant under translation by $L_{\Sigma}$. 
In this section, we will review descriptions of graded rings of piecewise polynomials $\PPoly(\Sigma)$ and Minkowski weights $\MW(\Sigma)$ associated to $\Sigma$. 
The motivation for these rings is that, if $\Sigma$ is also unimodular and $L_{\Sigma} = \{ 0 \}$, then $\MW(\Sigma)$ is  isomorphic to the integral  cohomology and $\PPoly(\Sigma)$ is isomorphic to the integral equivariant cohomology 
of the toric variety associated to $\Sigma$. However, our presentation will be purely combinatorial. Good sources for this material include \cite{ FS, ArdilaIntersection}.
One thing that we emphasize in our presentation is that we permit the case where $\dim L_{\Sigma} > 0$.

\subsection{The ring of piecewise polynomial functions}
Let $R$ be the ring of polynomial functions on $V$ with integral coefficients; we can also think of $R$ as $\Sym(V^{\vee}_\ZZ)$,
 and we write $R_d$ or $\Sym^d(V^{\vee}_\ZZ)$ for the degree $d$ part of $R$. Let $\PPoly(\Sigma)$ be the ring of continuous functions $f:V \to \RR$ such that, for each cone $\sigma$ of $\Sigma$, the restriction of $f$ to $\sigma$ is equal to the restriction of a polynomial in $R$.
The ring $\PPoly(\Sigma)$ is graded, where the degree $d$ part, $\PPoly^d(\Sigma)$,  is those piecewise polynomial functions $f$ such that $f|_{\sigma}$ is in $R_d$ for each cone $\sigma$ of $\Sigma$. 

We write $\Sigma_d$ for the set of $d$-dimensional cones of $\Sigma$, and $\Sigma_{\max}$ for $\Sigma_{\dim V}$.
 Since the maximal cones of $\sigma$ cover $V$, an element of $\PPoly(\Sigma)$ is determined by its restrictions to the maximal cones of $\Sigma$. 
In other words, $\PPoly(\Sigma)$ embeds into $R^{\Sigma_{\max}}:=\bigoplus_{\sigma \in \Sigma_{\max}} R$,
the ring of tuples of elements of~$R$ indexed by $\sigma\in\Sigma_{\max}$ with coordinatewise operations.
The next lemma describes the image of this embedding.

For any pair of cones $\sigma_1$ and $\sigma_2$ in $\Sigma_{\max}$ such that $\dim(\sigma_1 \cap \sigma_2) = \dim V - 1$, we define $\nu(\sigma_1, \sigma_2)$ to be a generator of $(\sigma_1 \cap \sigma_2)^\perp\cap V^{\vee}_\ZZ$; the vector $\nu(\sigma_1, \sigma_2)$ is defined up to sign.
\begin{lemma}
Let $(f_{\sigma})_{\sigma \in \Sigma_{\max}}\in R^{\Sigma_{\max}}$. Then the polynomials $f_{\sigma}$ glue to a continuous function in $\PPoly(\Sigma)$ if and only if, for each $\sigma_1$ and $\sigma_2$ such that $\dim(\sigma_1 \cap \sigma_2) = \dim V - 1$, we have $f_{\sigma_1} \equiv f_{\sigma_2} \bmod \nu(\sigma_1, \sigma_2) R$. 
\end{lemma}

The elements of $V^{\vee}$ give (global) linear functions on $V$. 
We establish some basic lemmas:
\begin{lemma}
Let $\Sigma_1$ and $\Sigma_2$ be two complete fans in finite dimensional vector spaces $V_1$ and $V_2$. We consider $\Sigma_1 \times \Sigma_2$ as a compete fan in $V_1 \oplus V_2$. Then the obvious isomorphism $\Sym(V_{1, \ZZ} \oplus V_{2, \ZZ}) \cong \Sym(V_{1, \ZZ}) \otimes \Sym(V_{2,\ZZ})$ induces an isomorphism
$\PPoly(\Sigma_1 \times \Sigma_2) \cong \PPoly(\Sigma_1) \otimes \PPoly(\Sigma_2)$.
\end{lemma}

Our final basic lemma of this subsection addresses refinement of fans.
\begin{lemma}
Let $\Sigma$ be a complete fan in $V$, and let $\Sigma'$ be a refinement of $\Sigma$. Then pullback of functions induces a ring homomorphism $\PPoly(\Sigma) \to \PPoly(\Sigma')$.
\end{lemma}

\subsection{Minkowski weights}
Let $\tau \subset \sigma$ be cones of $\Sigma$, with $\dim \tau = d-1$ and $\dim \sigma = d$. 
Then $(V_{\ZZ} \cap \Span(\sigma)) / (V_{\ZZ} \cap \Span(\tau))$ is a free $\ZZ$-module of rank $1$ inside $V_{\ZZ} / (V_{\ZZ} \cap \Span(\tau))$; let $\rho(\sigma, \tau)$ be the generator of $(V_{\ZZ} \cap \Span(\sigma)) / (V_{\ZZ} \cap \Span(\tau))$ which points into the cone $\sigma$.

Let $\mu$ be a function from $\Sigma_d$ to $\ZZ$. We call $\mu$ a \newword{$d$-dimensional Minkowski weight} if, for every $d-1$ dimensional cone $\tau$ of $\sigma$, we have the following equality in the quotient $V/\Span(\tau)$:
\begin{equation} \sum_{\sigma \supset \tau} \mu(\sigma) \rho(\sigma, \tau) = 0. \label{balancing} \end{equation}
Eq.~\ref{balancing} is called the \newword{balancing condition}.
We define a $d$-dimensional Minkowski weight $w$ to be \newword{nonnegative} if $w(\sigma) \geq 0$ for each $d$-dimensional cone $\sigma$.
We write $\MW_d(\Sigma)$ for the vector space of $d$-dimensional Minkowski weights and we write $\MW(\Sigma) = \bigoplus_d \MW_d(\Sigma)$. 

Minkowski weights behave simply under refinement of fans, as described by the following lemma:
\begin{lemma} \label{MinkowskiRefine}
Let $\Sigma$ be a complete fan in $V$, let $\Sigma'$ be a refinement of $\Sigma$, and let $\mu$ be a $d$-dimensional Minkowski weight on $\Sigma$. Define a function $\mu'$ on the $d$-dimensional cones of $\Sigma'$ as follows: If $\sigma'$ is contained in a $d$-dimensional cone $\sigma$ of $\Sigma$, then $\mu'(\sigma') = \mu(\sigma)$, otherwise $\mu'(\sigma')=0$. Then $\mu'$ is a $d$-dimensional Minkowski weight on $\Sigma'$.
\end{lemma}

\begin{lemma}
Let $\Sigma$ be a complete fan in $V$. The vector spaces $\MW_{\dim V}(\Sigma)$ and $\MW_{\dim L_{\Sigma}}(\Sigma)$ are one-dimensional.  Their respective generators are the Minkowski weight $[\Sigma]$ which is $1$ on every maximal cone of $\Sigma$, and the Minkowski weight $[L_{\Sigma}]$ which is $1$ on $L_{\Sigma}$. 
\end{lemma}

The above choice of generator for $\MW_{\dim L_{\Sigma}}(\Sigma)$ fixes  an isomorphism with $\ZZ$. Throughout the next sections, we will implicitly apply this isomorphism for Minkowski weights in $\MW_{\dim L_{\Sigma}}(\Sigma)$ and identify them with integers.

The vector space of Minkowski weights carries a product given by the  ``fan displacement rule" of Fulton and Sturmels~\cite{FS}. We use the tropical geometer's name \newword{stable intersection} for this  product operation on Minkowski weights,
and write it $\capstab$. We will not define stable intersection but refer to \cite[Theorem 3.2]{FS} or \cite[Section 2]{KatzTropInt} for more details. The tropical point of view on stable intersection shows that it is invariant under refinements. 
Stable intersection gives the following map making $\MW(\Sigma)$ into a ring:
$$\capstab \colon \MW_{d_1}(\Sigma) \times \MW(\Sigma)_{d_2} \to \MW_{d_1 + d_2- \dim V}.$$
We remark that the Minkowski weight $[\Sigma]$ acts as the identity element with respect to the operation  $\capstab$.

We now  simplify to the  case that $\Sigma$ is a complete unimodular fan, meaning that 
for any $d$-dimensional cone $\sigma$, 
the monoid $\sigma\cap V_\ZZ$ is generated by $L_\Sigma\cap V_\ZZ$ 
together with $d-\dim L_\Sigma$ further generators which form part of a lattice basis of~$V_\ZZ$.

We define a map $\iota^{\ast} : \PPoly^k(\Sigma) \to \MW_{\dim V - k}(\Sigma)$ which is  respectful of the gradings in the sense that $\iota^{\ast} : \PPoly^k(\Sigma) \to \MW_{\dim V - k}(\Sigma).$  
For a maximal cone $\sigma$ generated in the sense above by $L_\Sigma$ and the lattice basis $v_1, \dots, v_{n-\dim L_\Sigma}$, 
define  the rational function $e_{\sigma} = \prod_{i = 1}^{n-\dim L_\Sigma} \frac{1}{v_i^{\ast}}$,
where $v_i^{\ast}$ is the linear function taking value $1$ on~$v_i$ and $0$ on~$L_\Sigma$ and on each $v_j$ with $j\ne i$. 
For each pair of faces $\tau \subset \sigma$, define the rational function $e_{\sigma, \tau}$ of degree $\dim \tau - \dim V $ to be $e_{\overline{\sigma}}$ where $\overline{\sigma} $ is the image of $\sigma$ in $V / \langle \tau \rangle$. More explicitly, we have  $e_{\sigma, \tau} = \prod_{v_i \not \in \tau} \frac{1}{v_i^{\ast}}$. We can now state the formula for $\iota^{\ast}$. 

\begin{definition}\label{KatzPayne} 
Let $f = (f_{\sigma}) \in \PPoly^k(\Sigma)$ and $\tau$ be a face of $\Sigma$ of  dimension $k$. Then the Minkowski weight $\iota^{\ast}f $ is
\[ (\iota^{\ast} f) (\tau) =  \sum_{\sigma \supset \tau} e_{\sigma, \tau} f_{\sigma}. \]
\end{definition}

It follows from \cite[Proposition 1.2]{KatzPayne} that the right hand side of the above formula simplifies to be an integer.

By \cite[Theorem 2.1 and 3.2]{FS}, the ring of Minkowski weights $\MW(\Sigma)$ is isomorphic to the Chow cohomology ring of the toric variety of $\Sigma$ when $\Sigma$ is complete. 
Recall that  $\PPoly(\Sigma)$ is isomorphic to the equivariant Chow cohomology of the toric variety of $\Sigma$.  By \cite[Theorem 1.4]{KatzPayne}, the map from Definition~\ref{KatzPayne}
 is isomorphic to the canonical non-equivariant restriction map which gives a ring homomorphism from  equivariant Chow cohomology to Chow cohomology. 

Thanks to  the map $\iota^{\ast}$ and the operation of stable intersection, we can view $\MW(\Sigma)$ as a $\PPoly(\Sigma)$-module with multiplication  
$$\ast \colon  \PPoly^k(\Sigma) \times \MW_d(\Sigma) \to \MW_{d-k}(\Sigma),$$
  given by $f \ast \mu = \iota^{\ast}f \capstab \mu.$ 
 The annihilator of $\MW(\Sigma)$ is precisely the ideal
  of $\PPoly(\Sigma)$ generated by the global linear functions.

We now note some straightforward properties of the product $\ast$. The next two propositions can be deduced  from combining Defnition~\ref{KatzPayne} and  properties of stable intersection. 

\begin{proposition}
The $\ast$ operation is compatible with refinement in the following sense: Let $\Sigma$ be a complete fan, let $\Sigma'$ be a refinement of $\Sigma$, let $f \in \PPoly(\Sigma)$ and $\mu \in \MW(\Sigma)$. Let $f'$ be $f$ considered as a piecewise polynomial function on $\Sigma'$, and let $\mu'$ be the pullback of $\mu$ to $\Sigma'$ as in Lemma~\ref{MinkowskiRefine}. Then $(f') \ast (\mu')$ is the pullback of $f \ast \mu$ to $\Sigma'$, as in Lemma~\ref{MinkowskiRefine}.
\end{proposition}

\begin{proposition} \label{astIsLocal}
Let $f \in \PPoly^i(\Sigma)$, let $\mu$ be a $d$-dimensional Minkowski weight, and let $\tau$ be a $(d-i)$-cone of $\Sigma$. Then $(f \ast \mu)(\tau)$ depends only on the values of $f$ and $\mu$ on the cones containing $\tau$.
\end{proposition} 

The $\ast$ multiplication is compatible with fan products in a straightforward way; the next proposition gives the details.

\begin{proposition}
Let $\Sigma_1$ be a complete fan in $V_1$ and let $\Sigma_2$ be a complete fan in $V_2$; we write $\Sigma_1 \times \Sigma_2$ for the product fan in $V_1 \oplus V_2$. Let $f \in \PPoly(\Sigma_1)$, $g \in \PPoly(\Sigma_2)$, $\mu \in \MW_{d_1}(\Sigma_1)$, $\nu \in \MW_{d_2}(\Sigma_2)$.   Let $h  \in  \PPoly(\Sigma_1 \times \Sigma_2) $ be defined so that $h_{\sigma_1 \times \sigma_2}$ is the product of the pullbacks of $f_{\sigma_1}$ and~$g_{\sigma_2}$, and let $\xi$ be the Minkowski weight in $\MW_{d_1 + d_2}(\Sigma_1 \times \Sigma_2) $ defined by $\xi (\sigma_1  \times \sigma_2) = \mu(\sigma_1) \nu(\sigma_2)$ if $\dim \sigma_i = d_i$ for $ i = 1, 2$ and  $\xi(\sigma_1  \times \sigma_2) =0$ otherwise. 
Then $(h \ast \xi)( \tau_1 \times \tau_2)  = (f \ast \mu)(\tau_1) \cdot (g \ast \nu)(\tau_2)$, where $\dim \tau_1  = \dim \sigma_1 - \deg f$ and $\dim \tau_2  = \dim \sigma_2 - \deg g$.
\end{proposition}

The next proposition follows directly from Definition~\ref{KatzPayne}
and the fact that $\Sigma$ has no cones of dimension less than $\dim L_{\Sigma}$. Theorem~\ref{multIsStabIntersect} holds since $[\Sigma]$ is the identity for the operation of stable intersection. 

\begin{proposition}
If $j > \dim V - \dim L_{\Sigma}$ and $f \in \PPoly^j(\Sigma)$, then $\iota^{\ast} f =  0 $ and it follows that $f \ast \mu=0$ for any Minkowski weight $\mu$.
\end{proposition}

\begin{theorem} \label{multIsStabIntersect}
Let $f \in \PPoly^i(\Sigma)$ and $g \in \PPoly^j(\Sigma)$. Let $\nu = f \ast [\Sigma]$ and let $\mu = g \ast [\Sigma]$. Then $(fg) \ast [\Sigma] = \nu \capstab \mu$.
\end{theorem}

The following key property of the multiplication $\ast$ follows immediately from the  description of stable intersection  in~\cite{FS} and Theorem~\ref{multIsStabIntersect}. 
\begin{cor} \label{capStabPositive}
With notation as in Theorem~\ref{multIsStabIntersect}, if $f \ast [\Sigma]$  and $g \ast [\Sigma]$ are nonnegative Minkowski weights, then $(fg) \ast [\Sigma]$ is a nonnegative Minkowski weight.
\end{cor}

\section{Tautological classes of a matroid and the Bergman fan}
%
%
%
%
%
%

\subsection{The permutahedral fan}\label{PermutahedralFanSection}
Let $E$ be a finite set with $n$ elements.
We put $V = \RR^E$, and recall the notations $e_S$ and $e_i$ from the introduction.
For each $i$, $j \in E$, we have a hyperplane $\{ z(i) = z(j) \}$ in $V$; these hyperplanes form a hyperplane arrangement with $n!$ regions, corresponding to the total orders on $E$. 
We define the \newword{permutahedral fan} $\Sigma_E$ in $V$ to be the unimodular complete rational fan whose maximal cones are the regions of this hyperplane arrangement, so the cones of $\Sigma_E$ are indexed by total preorders on $E$.
The lineality space of $\Sigma_E$ is $1$-dimensional, and spanned by $e_E$. 

Cones of the permutahedral fan are of the form $\Span^+(e_{S_1}, e_{S_2}, \ldots, e_{S_k}) + \RR e_E$ where $\emptyset \subsetneq S_1 \subsetneq S_2 \subsetneq \cdots \subsetneq S_k \subsetneq E$ is a chain of nonempty proper subsets of $E$; the dimension of this cone is $k+1$. 
We will denote this cone by 
$\Cone(S_1, S_2, \ldots, S_k)$ or $\Cone(S_{\bullet})$. 
Putting $E_l = S_l \setminus S_{l-1}$, where $S_0 = \emptyset$ and $S_{k+1} = E$, the vector $z$ is in $\Cone(S_{\bullet})$ if and only if $z$ is constant on each $E_l$ and $z(E_1) \geq z(E_2) \geq \cdots \geq z(E_{k+1})$. 

\begin{remark}
We have chosen our sign conventions to match, for example,~\cite{BEST, AHK}. They are opposite to~\cite{AK}.
\end{remark}

\subsection{Tautological classes}\label{TautologicalClassesSection}
Let $M$ be a rank $r$ matroid on $E$. 
Let $z$ be a point in $V$ with generic coordinates.
We define the \newword{$z$-maximal basis} of $M$ to be the basis $(b_1, b_2, \ldots, b_r)$ of $M$ which maximizes $\sum_{i=1}^r z(b_p)$. 
Sorting $(b_1, b_2, \ldots, b_r)$ so that $z(b_1) > z(b_2) > \cdots > z(b_r)$, we define $x_p(z) := z(b_p)$.
By the greedy property of matroids \cite[Chapter 1.8]{Oxley}, we have 
\begin{equation} x_p(z) = \max {\Big\{} t : \rank_M \{ e : z(e) \geq t \} \geq p {\Big\}} . \label{xGlobal} \end{equation}
We use eq.~\ref{xGlobal} to define $x_p$ as a function on all of $V$, including at the points where some of the coordinates of $z$ coincide.
This clearly gives an element $x_p \in \PPoly^1(\Sigma_E)$.
When we need to indicate the matroid $M$, we'll write $x^M_p$.

We define the \newword{$z$-minimal cobasis} of $M$ to be the complementary set $(c_1, c_2, \ldots, c_{n-r}) = E \setminus \{ b_1, b_2, \ldots, b_r \}$, sorted so that $z(c_{n-r}) > z(c_{n-r-1}) > \cdots > z(c_2) > z(c_1)$, and we define $y_q(z) := -z(c_q)$. 
Again, we can alternately define
\begin{equation} y_q(z) = \min {\Big\{} t : \corank_M \{ e : z(e) < -t \} \leq n-q {\Big\}}  \label{yGlobal} \end{equation}
where $\corank_M(S) = \#(S) - \rank_M(S)$, so $y_q \in \PPoly^1(\Sigma_E)$.
Similarly, we'll write $y^M_q$ when we need to indicate the matroid $M$.

In recent literature, starting from \cite[Appendix III]{BEST}, these piecewise linear functions $x_p$ and $y_q$ tend to be spoken of as Chern roots of matroid tautological classes, which we introduce now.
For each matroid $M$, 
\cite{BEST} defines two \newword{tautological classes} $\mathcal S_M$ and $\mathcal Q_M$, and dual classes $\mathcal S^{\ast}_M$ and $\mathcal Q^{\ast}_M$, 
as equivariant $K$-classes of the permutahedral toric variety,
which generalize classes obtained from tautological vector bundles on the Grassmannian in the realizable case. 
The Chern classes of the tautological classes can be interpreted as piecewise polynomials on $\Sigma_E$,
and this is how we will formulate and initially use them.  
Let $\be_k$ be the $k$-th elementary symmetric function. 
Then we define two families of tautological (Chern) classes:
we let $c_k(\mathcal S^{\ast})$ be the element $\be_k(x_1, \ldots, x_r)$ in $\PPoly^k(\Sigma_E)$, and likewise let $c_k(\mathcal Q) = \be_k(y_1,\ldots,y_{n-r})$. 
These two are the ``positive'' families (e.g., $c_1(\mathcal S^{\ast})$ and $c_1(\mathcal Q)$ are nef);
in this paper we will not make much use of their ``negative'' duals,
but for the reader's ease of comparison, in our notation these would be
$c_k(\mathcal S)=\be_k(-x_1, \ldots, -x_r)$ and
$c_k(\mathcal Q^{\ast}) = \be_k(-y_1,\ldots,-y_{n-r})$.
Later we will use a different description of the tautological classes in terms of Minkowski weights, also provided in~\cite{BEST}.

We note that, assuming $M$ is loop-free, we have $x_1(z) = \max_{e \in E} z(e)$ and, assuming that $M$ is coloop-free, we have $y_1(z) = - \min_{e \in E} z(e)$.
We set $\alpha = - \min_{e \in E} z(e)$ and $\beta =  \max_{e \in E} z(e)$ (regardless of whether or not $M$ has loops or co-loops). 

\begin{remark}
We explain the relationship between our definition of $\alpha$, and the formulas which the reader will find in, for example,~\cite[Section 1.1]{ArdilaIntersection}, \cite[Section 5]{HuhKatz} or \cite[Section 5]{AHK}.
The function $\alpha  = - \min_{e \in E} z(e)$ is not constant on $\Sigma_L$, so it does not pull back from a function on  $\RR^E/L_{\Sigma_E}$. If we insist on working with fans that have trivial lineality space, then we need to replace $\alpha$ with a function that differs from $\alpha$ by an element of $V^{\vee}$ and is zero on $\Sigma_L$.
The usual way to do this is to choose some $i \in E$ and define $\alpha_i = z(i) + \alpha = \max_e(z(i) - z(e))$.
This maximum is closely related to the maximum in~\cite{HuhKatz}.
These papers also sometimes introduce functions $x_S$ indexed by $S \in 2^E$, and defined as the unique piecewise linear function on $\Sigma_E$ with $x_S(e_S) = 1$ and $x_S(e_T)=0$ for $S \neq T$. 
Note that we have $\alpha_i(e_S) = 1$ if $i \in S$ and $\alpha_i(e_S)=0$ if $i \not\in S$. So $\alpha_i = \sum_{S \ni i} x_S$, a formula which the reader will find in~\cite{ArdilaIntersection} or~\cite{AHK}.
\end{remark}

We write $\Crem$ for the negation map on $\RR^E$, which is the ``tropical Cremona transform". 
We have $y^{M^{\perp}}_q = \Crem^{\ast} x^M_q$ and $x^{M^{\perp}}_p = \Crem^{\ast} y^M_p$.

\subsection{The $x$ and $y$ functions in the neighborhood of a cone of the permutahedral fan}
Let $z$ be a point of $\RR^E$, lying in the relative interior of the cone  $\Cone(S_1, S_2, \ldots, S_k)$.
Put $E_i = S_i \setminus S_{i-1}$, where $S_0 = \emptyset$ and $S_{k+1} = E$. 
Recall that the cones of the permutahedral fan $\Sigma_E$ are of the form  $\Cone(S_{\bullet})$ where $S_{\bullet}$ is a flag of subsets of $E$. 
We have $\Cone(S_{\bullet}) \subset \Cone(S'_{\bullet})$  if the preorder on $E$ determined by $S'_{\bullet}$ refines the preorder of~$S_{\bullet}$.
Therefore, the star of $\Cone(S_{\bullet})$ in $\Sigma_E$  is isomorphic to $\prod_{i=1}^{k+1} \Sigma_{E_i} \subset \prod_{i=1}^{k+1} \RR^{E_i} = \RR^E$. 

This isomorphism provides a map $\rho_{S_{\bullet}}:\PPoly(\Sigma_E) \to \PPoly(\prod_{i=1}^{k+1} \Sigma_{E_i})$.
Given $f\in\PPoly(\Sigma_E)$, 
we let $\rho_{S_{\bullet}}(f)$ be the unique piecewise polynomial on the star of~$\Cone(S_{\bullet})$
that agrees with~$f$ on the union of the cones of~$\Sigma_E$ containing $\Cone(S_{\bullet})$.
We now want to describe the images of $x_p$ and~$y_q$ under~$\rho_{S_{\bullet}}$.
This is essentially the content of \cite[Proposition 5.3]{BEST}, but that proposition does not mention the individual Chern roots $x_p$ and $y_q$, so we re-prove it here in our notation.

Let $M_i$ be the matroid $M|{S_i} / S_{i-1}$, so $M_i$ is a matroid on the ground set $E_i$. We set $n_i = \#(E_i)$ and $r_i = \rank(M_i)$. 
So $\sum r_i = r$ and $\sum (n_i-r_i) = n-r$.

\begin{lemma} \label{restrictXY}
Let $a$ be the unique index such that $\sum_{i=1}^{a-1} r_i < p \leq \sum_{i=1}^a r_i$ and let $p' = p-\sum_{i=1}^{a-1} r_i$. Then
$\rho_{S_{\bullet}}(x_p^M)$ is the function $x_{p'}^{M_a}$, pulled back from the factor $\Sigma_{E_a}$ of~$\prod_i \Sigma_{E_i}$.

Likewise,  let $b$  be the unique index such that $\sum_{j=b+1}^{k+1} (n_j-r_j) < q \leq \sum_{j=b}^{k+1} (n_j-r_j)$ and let $q' = q-\sum_{j=b+1}^{k+1} (n_j-r_j)$. Then 
$\rho_{S_{\bullet}}(y_q^M)$ is the function $y_{q'}^{M_b}$, pulled back from the factor $\Sigma_{E_b}$. 
\end{lemma}

\begin{proof}
We verify the formula for $x$; the formula for $y$ is analogous.
Recall that, on the relative interior of the cone $\Cone(S_{\bullet})$, the coordinate function $z:E\to\mathbb R$ is constant on each set $E_i$ and $z(E_1) > z(E_2) > \cdots > z(E_{k+1})$. Thus, near such a point, the value(s) of $z$ on $E_i$ are greater than the values of $z$ on $E_j$, for $i <j$.

For such a nearby $z$, let $B=\{ b_1, b_2, \ldots, b_r \}$ be the $z$-maximal basis and let $B_i = B \cap E_i$. Thus, when $i$ equals the index $a$ in the statement of the lemma, $B_a$ is the $z|_{E_a}$-maximal basis of $M_a$. Since $B_i \subseteq E_i$, the value(s) of $z$ on $B_i$ are greater than the value(s) of $z$ on $B_j$ for $i<j$. We have $|E_i| = r_i$. Thus, $B_a = \{ b_{\sigma+1}, b_{\sigma+2}, \ldots, b_{\sigma+r_a} \}$ for $\sigma = \sum_{i=1}^{a-1} n_i$. So, for $\sigma+1 \leq p \leq \sigma+r_a$, we have $x_p^M = x_{p - \sigma}^{M_a}$ near $\Cone(S_{\bullet})$.
\end{proof}

\subsection{The Bergman fan and the Minkowski weights of the tautological functions}
The Bergman fan of a matroid $M$, denoted $\Berg(M)$, was introduced to study the tropicalization of linear spaces.
We interpret it as a Minkowski weight on~$\Sigma_E$, following Ardila and Klivans~\cite{AK};
it was earlier described with a coarser fan structure by Feichtner and Sturmfels~\cite{FeichtnerSturmfels}.

The Bergman fan is an $r$-dimensional Minkowski weight: On a cone $\Cone(S_1, S_2, \ldots, S_{r-1})$, the Bergman fan is $1$ if all the $S_i$ are flats of $M$, and $0$ otherwise.
Note that, since $M$ has rank $r$ and we have $\emptyset \subsetneq S_1 \subsetneq S_2 \subsetneq \cdots \subsetneq S_{r-1} \subsetneq E$, this implies that $S_i$ is a flat of rank $i$.

We can describe the support of the Bergman fan in another way. 
Resuming the notation from before Lemma~\ref{restrictXY},
with $S_{\bullet}$ a flag of sets determined by $z\in V$ and $M_i$ associated minors of~$M$,
denote the matroid $\bigoplus_i M_i$ as $\gr^z(M)$ or $\gr^{S_{\bullet}}(M)$.
Observe that the bases of $\gr^z(M)$ are the bases $B$ of $M$ which maximize $\sum_{b \in B} z(b)$. 
Then $z$ is in $\Berg(M)$ if and only if the matroid $\gr^z(M)$ is loop-free.

\begin{example} \label{U23Example}
Consider the uniform matroid of rank $2$ on the ground set $\{1,2,3 \}$. 
The support of $\Berg(M)$ is $\Cone(\{ e_1 \}) \cup \Cone( \{ e_2 \}) \cup \Cone( \{ e_3 \})$.
We have $\Cone( \{ 1 \}) = \{ z_1 \geq z_2=z_3 \}$, etcetera.
Let $w = (w_1, w_2, w_3)$ be a point in the relative interior of $\Cone(\{ 1 \})$.
 The associated graded matroid $\gr^w(M)$ is $M|{\{ 1 \}} \oplus M/\{1\} = U(1,\{ 1 \}) \oplus U(1, \{ 2,3 \})$, which is loop-free.
 The bases of $\gr^w(M)$ are $12$ and $13$, with weights $z_1+z_2 = z_1+z_3$; note that $23$ is not a basis of $\gr^w(M)$, and $z_2 + z_3 < z_1 +z_2 = z_1 + z_3$.
\end{example}

The description of the support of the Bergman fan helps motivate the following definition: For $\ell \geq 0$, we define the \newword{$\ell$-thickened Bergman fan} to be a Minkowski weight supported on the set of $z$ such that $\gr^z(M)$ has $\leq \ell$ loops. More precisely, $\Berg^{\ell}(M)$ is an $r+\ell$ dimensional Minkowski weight. Its value on $\Cone(S_1, S_2, \ldots, S_{r+\ell-1})$ is $1$ if $\gr^{S_{\bullet}(M)}$ has $\leq \ell$ loops, and is $0$ otherwise.
We next describe the thickened Bergman fan in terms of the lattice of flats. 

\begin{lemma}
Let $S_1 \subset S_2 \subset \cdots \subset S_{r+\ell-1}$ be a chain of proper nonempty subsets of $E$, and put $S_0 = \emptyset$ and $S_{r+\ell} = E$. 
Then $\gr^{S_{\bullet}(M)}$ has $\leq \ell$ loops if and only if exactly $\ell$ minors $M|S_{i+1}/S_i$ are single loops and the rest of the minors are rank~$1$ uniform matroids. 
In particular, $\gr^{S_{\bullet}(M)}$ has exactly $\ell$ loops in this case. 
\end{lemma}

\begin{proof}
One direction is obvious. 
For the other, suppose that $\gr^{S_{\bullet}(M)}$ has $\leq \ell$ loops. 
Let $\rank$ be the rank function of~$M$. 
Then 
\[0 \leq \rank(S_1) \leq  \rank(S_2)  \leq \cdots \leq \rank(S_{r+\ell-1}) \leq r,\]
and at least $\ell$ of these inequalities are equalities. 
If $\rank(S_{i}) = \rank(S_{i+1})$, then every element in $S_{i+1} \backslash S_i$ is a loop of $\gr^{S_{\bullet}(M)}$. 
So in order to have $\leq \ell$ loops, we must have that exactly $\ell$ of the inequalities are equalities, 
and if $\rank(S_{i})  =   \rank(S_{i+1})$ then $|S_{i+1}\backslash S_i| = 1$. This moreover implies that if $\rank(S_{i})  \neq \rank(S_{i+1})$, then $\rank(S_{i})  =   \rank(S_{i+1}) -1$. 

Therefore, the matroid $M|S_{i+1}/S_i$ is a single loop in the $\ell$ cases when $\rank(S_{i})  =   \rank(S_{i+1})$, and in the remaining cases where $\rank(S_{i})  =   \rank(S_{i+1}) -1$, the matroid $M|S_{i+1}/S_i$ has rank~$1$. If it is not uniform, it contains a loop, contradicting $\gr^{S_{\bullet}(M)}$ having  $\leq \ell$ loops. This proves the lemma. 
\end{proof}

Following the above lemma, the connection between the thickened Bergman fan and the tautological classes is provided by~\cite[Proposition 7.4]{BEST}.
\begin{theorem} \label{ChernIsThickBerg}
We have $c_{n-r-\ell}(\mathcal Q) \ast [\Sigma_E]=\Berg^{\ell}(M)$ and $c_{r-\ell}(\mathcal S^{\ast}) \ast [\Sigma_E]=\Crem^{\ast} \Berg^{\ell}(M^{\perp})$.
\end{theorem}

\begin{example}
We continue Example~\ref{U23Example}, so let $M$ be the uniform matroid of rank $2$ on the ground set $\{1,2,3 \}$. 
Then $x_1 = \max(z_1, z_2, z_3)$, $x_2 = \text{median}(z_1, z_2, z_3)$ and $y_1 = - \min(z_1, z_2, z_3)$.
The piecewise linear function $y_1$ is linear on the connected components of $\RR^3 \setminus \Berg(M)$, consistent with our claim that $y_1 \ast [\Sigma_{[3]}] = \Berg(M)$. 
\end{example}

We will also need the following key formula, which is very close to a formula conjectured in~\cite{LdMRS} and proved in~\cite[Theorem 10.12]{BEST}.
 Recall the notation $\be_k$ for the $k$-th elementary symmetric function and the notation $c(M)$ for the number of connected components of $M$.
 Recall that $\MW_1(\Sigma_E) \cong \RR$, with generator given by the Minkowski weight $[L_{\Sigma_E}]$; we identify $\MW_1(\Sigma_E)$ with $\RR$ in the following. 
\begin{theorem} \label{KrisFormula}
Let $M$ be a  co-loop free 
matroid with $n \geq 2$ elements.  We have:
\[ g_i(M) = (-1)^{c(M)-1} {\Big(} \be_{i-1}(y_2, y_3, \ldots, y_{n-r}) \ast \Berg(M) {\Big)} \capstab {\Big(}\Crem^{\ast} \Berg^{i}(M^{\perp}) {\Big)} . \]
\end{theorem}

\begin{proof}
Recall the notation $\alpha$ for the function $\alpha(z) = - \min_{e \in E} z(e)$; this is an element of  $\PPoly^1(\Sigma_E)$.
The following formula is \cite[Theorem 10.12]{BEST}: 
\[ g(s) = (-1)^{c(M)} \sum_i \sum_j \left( \alpha^j c_{i-j-1}(\mathcal Q^{\ast}) c_{r-i}(\mathcal S^{\ast}) c_{n-r}(\mathcal Q) \right) \ast [\Sigma_E] (-s)^i .\]
We have edited the formula from how it appears in~\cite{BEST} to avoid notation we have not defined, replacing ``$\deg_{\alpha}$" with its definition as a sum over powers of $\alpha$. 
We've also replaced ``$|E|$" with ``$n$".

We now extract the coefficient $g_i(M)$ of $s^i$. Because $M$ is co-loop free, we have $\alpha = y_1$:
\[
g_i(M) = (-1)^{c(M)+i} \sum_j  \left( y_1^j c_{i-j-1}(\mathcal Q^{\ast}) c_{r-i}(\mathcal S^{\ast}) c_{n-r}(\mathcal Q) \right) \ast [\Sigma_E] 
\]
Recalling the definition of the tautological class $c_k(\mathcal Q^{\ast})$ from Section~\ref{TautologicalClassesSection}, this is
\begin{align*}
g_i(M) &=  \sum_j (-1)^{c(M)+i-j} \left((-1)^j y_1^j \be_{i-j-1}(-y_1, -y_2, \ldots, -y_{n-r}) c_{r-i}(\mathcal S^{\ast}) c_{n-r}(\mathcal Q) \right) \ast [\Sigma_E]
\\ &= \sum_j (-1)^{c(M)-1} \left((-1)^j y_1^j \be_{i-j-1}(y_1, y_2, \ldots, y_{n-r}) c_{r-i}(\mathcal S^{\ast}) c_{n-r}(\mathcal Q) \right) \ast [\Sigma_E]. 
\end{align*}
We next use the identity $\sum_{j\ge0} (-1)^j  y_1^j \be_{i-j-1}(y_1, y_2, \ldots, y_{n-r}) = \be_{i-1}(y_2, y_3, \ldots, y_{n-r})$,
which can be checked by comparing the coefficient of any monomial on both sides.
The coefficient is visibly zero unless the monomial has the form $y_1^k y_{\ell_1}\cdots y_{\ell_{i-1-k}}$ for $\ell_1,\ldots,\ell_{i-1-k}\subseteq\{2,\ldots,n-r\}$ distinct.
If $k>0$, this monomial appears just twice on the left with opposite signs, for $j=k$ and $j=k-1$, and is absent on the right,
whereas if $k=0$ it has coefficient 1 on the left (at $j=0$) and on the right.
With this, the last displayed sum simplifies to
\[ (-1)^{c(M)-1} {\Big(} \be_{i-1}(y_2, y_3, \ldots, y_{n-r})  c_{r-i}(\mathcal S^{\ast}) c_{n-r}(\mathcal Q) {\Big)} \ast  [\Sigma_E] .\]
Using Theorem~\ref{multIsStabIntersect}, this is
\[ (-1)^{c(M)-1} {\Big(} \be_{i-1}(y_2, y_3, \ldots, y_{n-r}) \ast c_{n-r}(\mathcal Q) \ast  [\Sigma_E] {\Big)} \capstab {\Big(} c_{r-i}(\mathcal S^{\ast}) \ast [\Sigma_E] {\Big)}. \]
We can evaluate two of the $\ast$-products using Theorem~\ref{ChernIsThickBerg}, to give
\[ (-1)^{c(M)-1} {\Big(} \be_{i-1}(y_2, y_3, \ldots, y_{n-r}) \ast \Berg(M) {\Big)} \capstab {\Big(}\Crem^{\ast} \Berg^{i}(M^{\perp}) {\Big)}, \]
which is the desired formula.
\end{proof}

We note how Theorem~\ref{KrisFormula} specializes to give a formula for $\omega(M)$:
\begin{cor} \label{OmegaFormula}
Let $M$ be a co-loop free matroid of rank $r$ with $n \geq 2$ elements. Then
\[ \omega(M) =  (-1)^{c(M)-1} \be_{r-1}(y_2, y_3, \ldots, y_{n-r}) \ast \Berg(M) . \]
\end{cor}

\begin{proof}
This is the case $r=i$ of Theorem~\ref{KrisFormula}. In this case, the second term in the stable intersection is $\Crem^{\ast} \Berg^{r}(M^{\perp})$. Since $M^{\perp}$ has rank $n-r$, the matroid $\gr^z(M^{\perp})$ has at most $r$ loops for any $z$, so $\Berg^{r}(M^{\perp})$ is supported on the whole of $\RR^E$. In other words, $\Berg^r(M^{\perp})$ equals $[\Sigma_E]$, as does its pullback by $\Crem$,
and stable intersection with $[\Sigma_E]$ acts as the identity.
\end{proof}

At the other extreme, the special case $i=1$ of Theorem~\ref{KrisFormula} gives the $\beta$-invariant.
We can drop the $(-1)^{c(M)-1}$ term in this case, since both sides of the equation are $0$ if $c(M)>1$.
We recover a formula for $\beta$ found in \cite{BEST},
which for complex realizable matroids is Varchenko's conjecture on critical points of a hyperplane arrangement
\cite{Varchenko} (see also \cite[Theorem 13]{CHKS}).
\begin{cor}[{\cite[Theorem 6.2]{BEST}}] \label{betaFormula}
For any co-loop free matroid $M$, we have
\[ \beta(M) = \Berg(M) \capstab \Crem^{\ast} \Berg^1(M^{\perp}). \]
\end{cor}

\begin{remark}
The intersection theoretic-formula for the beta invariant in Corollary~\ref{betaFormula} 
is very similar to, but not the same as, two previous results.
We first explain the main result of~\cite{AEP}:
Let $i \in E$ and let $\pi : \RR^E \to \RR^{E \setminus i}$ be the linear projection that forgets the $i$ coordinate. There are pullback and pushforward maps on Minkowski weights induced by $\pi$. Then \cite{AEP} shows that
\[ \beta(M) = \Berg(M) \capstab \Crem^{\ast} \pi^{\ast} \pi_{\ast} \Berg(M^{\perp}). \]
Both $\Berg^1(M^{\perp})$ and $\pi^{\ast} \pi_{\ast} \Berg(M^{\perp})$ are thickenings of $\Berg(M^{\perp})$ to a fan of one dimension larger, but they are not the same thickening. Additionally, the main result of \cite{RauPH} shows that  the beta invariant of a matroid is, up to sign, the self-intersection of the diagonal class in $\Berg(M) \times \Berg(M)$. 
It would be interesting to understand better the relationship between Corollary~\ref{betaFormula} and the results of \cite{RauPH} and~\cite{AEP}.
\end{remark}

\section{Proof of {Theorem~\ref{OmegaPosReduction}}}\label{ProofOPR}

In this section, we will prove Theorem~\ref{OmegaPosReduction}.
We start with some preliminary reductions.
First of all, if $M$ has a loop or co-loop, then $g_M(t)=0$, so there is nothing to prove.
Next, if $M = \bigoplus M_i$ for various connected matroids $M_i$ with at least two elements each, then $g_M(t) = \prod g_{M_i}(t)$, so if all the $g_{M_i}(t)$ have nonnegative coefficients, then so does $g_M(t)$.
Thus, it is enough to prove Theorem~\ref{OmegaPosReduction} in the case that $M$ is connected and has at least two elements.

By Theorem~\ref{KrisFormula}, we have
\[ g_i(M)=  {\Big(} \be_{i-1}(y_2, y_3, \ldots, y_{n-r}) \ast \Berg(M) {\Big)} \capstab {\Big(}\Crem^{\ast} \Berg^{i}(M^{\perp}) {\Big)}. \]
Since $\Crem^{\ast} \Berg^{i}(M^{\perp})$ is a positive Minkowski weight, by Corollary~\ref{capStabPositive}, we will be done if we can show that the Minkowski weight $\be_{i-1}(y_2, y_3, \ldots, y_{n-r}) \ast \Berg(M)$ is likewise positive. 
Thus, let $S_1 \subset S_2 \subset \cdots \subset S_{r-i}$ be a chain of proper nonempty subsets of $E$, our goal is to prove that the Minkowski weight $\be_{i-1}(y_2, y_3, \ldots, y_{n-r}) \ast \Berg(M)$ is nonnegative on $\Cone(S_{\bullet})$. 
Since $\be_{i-1}(y_2, y_3, \ldots, y_{n-r}) \ast \Berg(M)$ is supported on a subfan of $\Berg(M)$, it is enough to verify the case where the $S_i$ are flats of $M$, and we therefore change the name of $S_j$ to $F_j$.
Thus, the key to our proof is to give a formula for the Minkowski weight $\be_{i-1}(y_2, y_3, \ldots, y_{n-r}) \ast \Berg(M)$ on $\Cone(F_{\bullet})$.

We first need to define more notation. 
We shorten $r-i+1$ to $p$, so our chain of flats is $F_1 \subset F_2 \subset \cdots \subset F_{p-1}$.
Put $F_0 = \emptyset$ and $F_p = E$.
Let $E_j = F_j \setminus F_{j-1}$, so $E = E_1 \sqcup E_2 \sqcup \cdots \sqcup E_p$. 
Let $M_j$ be the matroid $M|{F_j}/F_{j-1}$; this is a matroid on the ground set $E_j$.
Put $n_j = \#(E_j)$ and $r_j = \rank(M_j)$, so $\sum n_j = n$ and $\sum r_j = r$. 

\begin{lemma} \label{keyLocalizationComputation}
Let $M$ be a connected matroid on $\geq 2$ elements.
With the above notations, the value of the Minkowski weight $\be_{i-1}(y_2, y_3, \ldots, y_{n-r}) \ast \Berg(M)$ on $\Cone(F_{\bullet})$ is
\[ \prod_{j=1}^{p-1}(-1)^{c(M_p)-1} \left( \Berg(M_j) \capstab \Berg^{n_j-2r_j+1}(M_j) \right) \times \omega(M_p) . \]
\end{lemma}

\begin{proof}
By Proposition~\ref{astIsLocal}, we can restrict our attention to the star of $\Cone(F_{\bullet})$ in $\Sigma_E$; this is isomorphic to $\prod_j \Sigma_{E_j}$. 
The images of the functions $y_q^M$ (for $1 \leq q \leq n-r$) under $\rho_{F_{\bullet}}$ 
are $y_{q'}^{M_j}$ for $1 \leq q' \leq n_j-r_j$, with the bijection between $[n-r]$ and $\bigsqcup_{j=1}^{r-i+1} [n_j-r_j]$ given by Lemma~\ref{restrictXY}.
In the particular case $q=1$, in the notation of Lemma~\ref{restrictXY} the index $b$ equals $k+1$, and $q'=q=1$.
This $k+1$ is the index of the final $E$ in the chain of subsets $F_{\bullet}$, which in the present notation is $p$;
that is, $y_1^M$ corresponds to $y_1^{M_p}$. 
We deduce that
\begin{multline*} \be_{i-1}(y_2^M, y_3^M, \ldots, y_{n-r}^M) = \\
 \sum_{s_1+s_2+\cdots+s_p= i-1} \left[ \left( \prod_{j=1}^{p-1}  \be_{s_j}(y_1^{M_j}, y_2^{M_j}, \ldots, y_{n_j-r_j}^{M_j}) \right) \times  \be_{s_p}(y_2^{M_p}, \ldots, y_{n_p-r_p}^{M_p}) \right]. \label{coproductE} \end{multline*}
 So the quantity that we want to compute is the sum of
\begin{equation}
 \prod_{j=1}^{p-1} \left(  \be_{s_j}(y_1^{M_j}, y_2^{M_j}, \ldots, y_{n_j-r_j}^{M_j}) \ast \Berg(M_j) \right) \times \left(  \be_{s_p}(y_2^{M_p}, \ldots, y_{n_p-r_p}^{M_p})  \ast \Berg(M_p) \right) \label{theSummand} \end{equation}
 over the index set $\{ (s_1, s_2, \ldots, s_p) : s_1+s_2+\cdots + s_p = i-1 \}$. Now, if we have $s_j > \dim \Berg(M_j) - \dim L_{\Sigma_{E_j}} = r_j-1$ for any index $j$, then the corresponding factor in eq.~\ref{theSummand} is $0$.
 But $\sum_{j=1}^p (r_j-1) = r - p = i-1$, so the only way to have $s_j \leq r_j-1$ for all $j$ is to have $s_j = r_j-1$ for all $j$. 
So our sum collapses to the single value
\[  \prod_{j=1}^{p-1} \left(  \be_{r_j-1}(y_1^{M_j}, y_2^{M_j}, \ldots, y_{n_j-r_j}^{M_j}) \ast \Berg(M_j) \right) \times \left(  \be_{r_p-1}(y_2^{M_p}, \ldots, y_{n_p-r_p}^{M_p})  \ast \Berg(M_p) \right).  \]

By Theorem~\ref{multIsStabIntersect} and Theorem~\ref{ChernIsThickBerg},
\begin{align*} 
&\mathrel{\phantom=}\be_{r_j-1}(y_1^{M_j}, y_2^{M_j}, \ldots, y_{n_j-r_j}^{M_j}) \ast \Berg(M_j) 
\\& = \left( \be_{r_j-1}(y_1^{M_j}, y_2^{M_j}, \ldots, y_{n_j-r_j}^{M_j}) \ast [\Sigma_{E_j}] \right) \capstab  \Berg(M_j) \\& = \Berg^{n_j-2r_j+1}(M_j) \capstab \Berg(M_j) , 
\end{align*}
giving the expected factor for $1 \leq j \leq p-1$. For the final factor, note that $M_p$ is $M/F_{p-1}$ and therefore $M_p$, like $M$, is co-loop free. Thus, Corollary~\ref{OmegaFormula} applies, and $\be_{r_p-1}(y_2^{M_p}, \ldots, y_{n_p-r_p}^{M_p})  \ast \Berg(M_p)$ is $(-1)^{c(M_p)-1} \omega(M_p)$. 
We have deduced the claim.
\end{proof}

We now conclude the proof of Theorem~\ref{OmegaPosReduction}. Since $\Berg^{n_j-2r_j+1}(M_j) \capstab \Berg(M_j)$ is an intersection of positive Minkowski weights, it is nonnegative.
The matroid $M_p$ is a quotient of $M$, so our hypothesis states that $\omega(M_p) \geq 0$. Thus, our product in Lemma~\ref{keyLocalizationComputation} is nonnegative. 
Our hypothesis that $M/F$ is connected for every proper $F$ means that, in particular, $M_p = M/F_{p-1}$ is connected, so $c(M_p) = 1$ and $(-1)^{c(M_p)-1}=1$.
We have established that $\be_{i-1}(y_2, y_3, \ldots, y_{n-r}) \ast \Berg(M)$ is a nonnegative Minkowski weight when $M/F$ is connected for all proper flats $F$.
As explained before, since $\Crem^{\ast} \Berg^{i}(M^{\perp})$ is likewise nonnegative, Theorem~\ref{KrisFormula} now implies that $g_i(M) \geq 0$. \hfill\qedsymbol

\begin{eg} \label{DisconnectedExample}
We add a basic example to illustrate what goes wrong when $M/F$ is disconnected. Let $M$ be the rank $3$ matroid on $[5]$ where $\{ 1, 2, 5 \}$ and $\{3, 4, 5 \}$ have rank $2$ and the points are otherwise general. This is a series-parallel matroid, so $g(t) = t$, and so $g_2(M) = g_3(M)=0$. Let us see what happens when we compute $g_2(M)$ using Theorem~\ref{KrisFormula} and Lemma~\ref{keyLocalizationComputation}. 

Theorem~\ref{KrisFormula} has two Minkowski weights: $(y_2+y_3) \ast \Berg(M)$ and $\Crem^{\ast} \Berg^{2}(M^{\perp})$. Call these $A$ and $B$.
The Minkowski weight $A$ is supported on the rays of the Bergmann fan. More specifically, it is $1$ on the rays corresponding to the flats $\{ 1,2,5 \}$ and $\{ 3,4,5 \}$ and $-1$ on the ray corresponding to the flat $\{ 5 \}$. Note that $M/\{ 5 \}$ is, indeed, disconnected, accounting for the minus sign. Note also that $(e_1+e_2+e_5) + (e_3+e_4+e_5) - e_5 = e_1+e_2+\cdots+e_5$ is in the lineality space of the Bergmann fan, so this Minkowski weight is balanced. To see that $A \capstab B = 0$, we will translate $A$ to miss $B$. Indeed, consider $A + (1,2,3,4,5)$, which is supported on the rays $(1+x,2+x,3,4,5+x)$, $(1,2,3+x,4+x,5+x)$ and $(1,2,3,4,5+x)$ for $x \geq 0$. Since $B$ is supported on the codimension one faces of the braid arrangement, the second and third ray of $A+(1,2,3,4,5)$ miss $B$ completely. The first ray intersects the  codimension one faces of the braid arrangement at $(2,3,3,4,6)$, $(3,4,3,4,7)$, $(4,5,3,4,8)$. None of these points are on $B$.

So Theorem~\ref{KrisFormula} gives $0$, as required, but the first Minkowski weight is not positive.
\end{eg}

\begin{remark}
The quantity $\Berg^{n_j-2r_j+1}(M_j) \capstab \Berg(M_j)$ played a key role in our proof. 
One combinatorial formula for $\Berg^{n-2r+1}(M) \capstab \Berg(M)$ is given in \cite[Corollary 6.16]{externalActivityPair}:
this intersection number is the number of bases of the diagonal Dilworth truncation $D(M,M)$
whose external activity statistic equals $(0, n-2r+1)$.
It would be interesting to further study the properties of this quantity and its relation to known matroid invariants.
\end{remark}

\addtocontents{toc}{\smallskip \textbf{Part 2:  Ferroni's formula and simplifications}}

\section{Schubert matroids}\label{SchubertMatroidsSection}
We return to reviewing background,
beginning by defining a collection of matroids, on a finite set $E$, which are called \newword{Schubert matroids}.
A Schubert matroid can be indexed by three kinds of data:
\begin{enumerate}
\item A total order $\le$ on $E$ and a subset $A$ of $E$.
\item A chain of subsets $\emptyset=S_0\subsetneq S_1\subsetneq \cdots \subsetneq S_{k}=E$ and a subset $A$ of $E$.
\item A chain of subsets $S_{\bullet}$ of $E$ as above and a tuple of nonnegative integers $(a_0, a_1, \ldots, a_k)$ with $a_0 = 0$, $a_k=r$ and 
\[ a_{i-1} \leq a_i \leq a_{i-1} + |S_i \setminus S_{i-1}| . \]
\end{enumerate}
We will denote the Schubert matroid as $\Omega_{\le, A}$, $\Omega_{S_{\bullet}, A}$ or $\Omega_{S_{\bullet}, a}$ correspondingly, allowing the type of the inputs to make clear which definition to apply.

Our first definition is taken from~\cite[Def 2.1]{Ferroni}. Let $\le$ be a total order on $E$.
This induces a partial order on the set of  subsets of~$E$ of any fixed size~$r$:
given two such subsets $A$, $B$, we say $A\le B$  if we have $a_i \leq b_i$ for all $1 \leq i \leq r$, where $A = \{ a_1, a_2, \ldots, a_r \}$ and $B = \{ b_1, b_2, \ldots, b_r \}$ with $a_1 \le a_2 \le \cdots \leq a_r$ and $b_1 \le b_2 \le \cdots \le b_r$.
An $r$-element subset $B$ of $E$ is defined to be a basis of $\Omega_{\le,A}$ if and only if $B \geq A$.

For our next definition, let $S_{\bullet}$ be a chain of sets with $\emptyset=S_0\subsetneq S_1\subsetneq \cdots \subsetneq S_{k}=E$ and let $A$ be a subset of $E$. Then an  $r$-element subset $B$ is a basis of $\Omega_{S_{\bullet}, A}$ if and only if 
\[ |B \cap S_i| \leq |A \cap S_i| \]
for each index $i$.

Finally, let $S_{\bullet}$ be a chain of subsets of~$E$ as above, and let $a = (a_0, a_1, \ldots, a_k)$ be a sequence of nonnegative integers obeying the inequalities above.
Then $B$ is a basis of $\Omega_{S_{\bullet}, a}$ if and only if $|B| = r$ and 
\[ |B \cap S_i| \leq a_i \ \text{for} \ 1 \leq i \leq k-1 . \]

\begin{prop}
The three definitions above define the same set of matroids.
\end{prop}

\begin{proof}
First, suppose that we start with a total order $e_1 < e_2 < \cdots < e_n$ on $E$ and a subset $A \subseteq E$. Let $S_i = \{ e_1, e_2, \ldots, e_i \}$. It is straightforward to check that $\Omega_{S_{\bullet}, A} = \Omega_{\leq, A}$.

Next, suppose that we start with a chain of sets $S_{\bullet}$ as above and a subset $A$ of $E$. Put $a_i = |A \cap S_i|$. We have $a_0 = |A \cap S_0| = |\emptyset| = 0$, $a_k = |A \cap S_k| = |A| = r$. Moreover, $a_i - a_{i-1} = |A \cap S_i| - |A \cap S_{i-1}| = |A \cap (S_i \setminus S_{i-1})|$, which must lie between $0$ and $|S_i \setminus S_{i-1}|$. So the vector $(a_0, \ldots, a_k)$ obeys the required inequalities, and it is immediate from the definitions that $\Omega_{S_{\bullet}, A} = \Omega_{S_{\bullet}, a}$. 

Finally, let $S_{\bullet}$ and $a=(a_0, \ldots, a_k)$ be a sequence of sets, and a sequence of nonnegative integers, obeying the required conditions. Let $\le$ be any total order on $E$ for which the $S_i$ are initial segments. Let $A_i$ be the $a_i - a_{i-1}$ first elements of $S_i \setminus S_{i-1}$ with respect to the order $\le$ and put $A = \bigcup A_i$.  Then $\Omega^{\le, A} = \Omega^{S_{\bullet}, a}$.
\end{proof}

\begin{remark}
None of our three choices of indexing data describes a Schubert matroid uniquely. 
If we use the indexing by pairs $(S_{\bullet}, a)$, and we further impose strict inequalities $a_{i-1} < a_i < a_{i-1} + |S_i \setminus S_{i-1}|$ for $1 \leq i \leq k$, then we obtain a unique label for each Schubert matroid. 
\end{remark}

We also make a triple of complementary definitions. We define $\Omega^{\le, A}$ to be the rank $r$ matroid where $B$ is a basis if $B \leq A$, so $\Omega^{\le, A} = \Omega_{\le_{\text{op}}, A}$, where $\le_{\text{op}}$ is the reversal of the order $\le$. 
Correspondingly, $\Omega^{S_{\bullet}, A}$ is $\Omega_{\overline{S}_{\bullet}, A}$ where $S_i = E \setminus S_{k-i}$, and
$\Omega^{S_{\bullet}, a} = \Omega_{\overline{S}_{\bullet}, \overline{a}}$ where $\overline{S}_i = E \setminus S_{k-i}$ and $\overline{a}_i = r - a_{k-i}$. 
With this, we describe the polyhedral  geometry of Schubert matroids. 
\begin{prop}
Let $S_{\bullet}$ and $a$ be a chain of sets and a nonnegative integer vector as above. Then the matroid polytopes of the corresponding Schubert matroids are given by
\begin{align*}
\Delta(\Omega_{S_{\bullet}, a}) &=  \Delta(r,E) \cap \bigcap_i \Big\{ \sum_{j \in S_i} z_j \leq a_i \Big\}, \\
\Delta(\Omega^{S_{\bullet}, a}) &=  \Delta(r,E) \cap \bigcap_i \Big\{ \sum_{j \in S_i} z_j \geq a_i \Big\} . 
\end{align*}
\end{prop}

We prove some lemmas which will be useful later:
\begin{prop} \label{SchubertLoopColoopConnected}
Let $X = \{ x_1, x_2, \ldots, x_n \} \subseteq E$, ordered as $x_1 < x_2 < \cdots < x_n$. Let $A \subseteq X$. If $x_1 \not\in A$, then $x_1$ is a loop of $\Omega_{\leq, A}$. If $x_n \in A$, then $x_n$ is a co-loop of $\Omega_{\leq, A}$. If $x_1 \in A$ and $x_n \not\in A$, then $\Omega_{\leq, A}$ is connected. 
\end{prop}

\begin{proof}
Let $A = \{ a_1, a_2, \ldots, a_r \}$ with $a_1 < a_2 < \cdots < a_r$. Let $B$ be any basis of $\Omega_{\leq, A}$ with $B = \{ b_1, b_2, \ldots, b_r \}$ and $b_1 < b_2 < \cdots < b_r$. 

Suppose that $x_1 \not\in A$. Then $a_1>x_1$, and we know that $b_1 \geq a_1$, so we conclude that $x_1 \not\in B$. We have shown that $x_1$ is not in any basis of $\Omega_{\leq, A}$, so $x_1$ is a loop.

Similarly, suppose that $x_n \in A$. Then $a_r = x_n$ and $b_r \geq a_r$ so $b_r = x_n$ as well. We have shown that $x_n$ is in every basis of $\Omega_{\leq, A}$, so $x_n$ is a coloop.

Finally, suppose that $x_1 \in A$ and $x_n \not\in A$. Put $B_0 = \{ x_1, x_{n-r+1}, x_{n-r+2}, \ldots, x_{n-2}, x_{n-1} \}$. So $B_0$ is a basis of $\Omega_{\leq, A}$. We note that the following sets are also bases of $\Omega_{\leq, A}$:
\[ \begin{array}{ll}
B_0 \setminus \{ 1 \} \cup \{ i \} & \text{for}\ 2 \leq i \leq n-r \\
B_0 \setminus \{ j \} \cup \{ n \} & \text{for}\ n-r+1 \leq j \leq n-1 \\
B_0 \setminus \{ 1 \} \cup \{ n \} & \\
\end{array} . \]
Therefore, the matroid polytope $\Delta(\Omega_{\leq, A})$ contains edges in the directions
\[ \begin{array}{ll}
e_{x_i} - e_{x_1} & \text{for}\ 2 \leq i \leq n-r \\
e_{x_n} - e_{x_j} & \text{for}\  n-r+1 \leq j \leq n-1 \\
e_{x_n} - e_{x_1} & \\ 
\end{array}. \]
These vectors span the vector space $\{ (z_1, z_2, \ldots, z_n) : \sum z_i = 0 \}$, so $\Delta(\Omega_{\leq, A})$  has dimension $n-1$ and $\Omega_{\leq, A}$ is connected.
\end{proof}

Recall that, for any covaluative invariant $v$, we define the valuative invariant $v^{\circ}$ by $v^{\circ}(M) = (-1)^{c(M)-1} v(M)$. 

\begin{cor} \label{OmegaValCoval}
Let $n \geq 2$. For any Schubert matroid $\Omega$ on $n$ elements, we have $\omega(\Omega) = \omega^{\circ}(\Omega)$.
\end{cor}

\begin{proof}
By Proposition~\ref{SchubertLoopColoopConnected}, either $\Omega$ has a loop, $\Omega$ has a coloop, or $\Omega$ is connected. In the first two cases, $\omega(\Omega) = 0$, so $\omega^{\circ}(\Omega)=0$ as well. In the last case, $c(\Omega) = 1$, so $\omega(\Omega) = \omega^{\circ}(\Omega)$.
\end{proof}

\section{Relations in the group of polyhedra} \label{GroupOfPolyhedraSection}
With $V$ and~$V^\vee$ as in Section~\ref{PiecewisePolynomialsSection}, 
let $\bOne(S):V^\vee\to\ZZ$ denote the indicator function of a set $S$. 
Let $\IndicatorGroup(V^\vee)$ denote the group generated by indicator functions of polyhedra in~$V^\vee$.
Let $P^\circ$ denote the relative interior of a polyhedron~$P$.

\begin{theorem}[{\cite[\S14]{McMullen}}]\label{McMullenEulerMap}
There is an endomorphism $\mbox{---}^*$ of $\IndicatorGroup(V^\vee)$, the \newword{Euler map}, with 
\[\bOne(P)^*= (-1)^{\codim P}\bOne(P^\circ)\]
for all polyhedra $P\subseteq V^\vee$.
\end{theorem}

A matroid function is valuative if and only if it factors through the map $M\mapsto\Delta(M)$ to~$\IndicatorGroup(\RR^E)$ \cite[Theorem 3.5]{DerksenFink}.
In this way the Euler map exchanges valuations and covaluations:
if a valuation $v$ satisfies $v(M)=\widehat v(\bOne(\Delta(M)))$ for a map $\widehat v$ with domain $\IndicatorGroup(\RR^E)$,
then $v^\circ(M)=\widehat v(-\bOne(\Delta(M)^*))$.
The minus sign accounts for the $-1$ in the exponent in the definition $v^\circ(M) = (-1)^{c(M)-1} v(M)$.

Define a preorder on polyhedra in~$V^\vee$ by $P\precsim Q$ if $P$ is a Minkowski summand of a dilate of~$Q$,
in other words, the normal fan of~$Q$ refines a subfan of the normal fan of~$P$. 
Given such a pair $P\precsim Q$ and a face $F$ of~$Q$, 
define the \newword{tangent cone} to $P$ at~$F$ to be
\[\Tangent_FP = \bigcup_{\lambda\ge0} P+\lambda\cdot(Q-f),\]
where $f$ is a point of the relative interior of~$F$.
Note that we do not translate the minimal face of the cone $\Tangent_FP$ to contain the origin.
The minimal face of $\Tangent_FP$ instead contains the face of~$P$ whose normal cone contains the normal cone of~$F$.
We also remark that $\Tangent_P P = V^{\vee}$. 

\begin{theorem}[Inward Brianchon--Gram formula]\label{BrianchonGram}
Let $P\precsim Q$ be polyhedra in $V^\vee$. Then
\[\bOne(P) = \sum_{\textnormal{\scriptsize $F$ a face of $Q$}}(-1)^{\dim F}\, \bOne(\Tangent_F P)\]
holds in $\IndicatorGroup(V^\vee)$.
\end{theorem}

\begin{eg} \label{LineSegmentEG}
Let $V^{\vee} = \RR$ and let $P$ be the line segment $[0,1]$. The inward Brianchon-Gram formula says that 
$\bOne([0,1]) = \bOne([0,\infty)\,) + \bOne(\,(-\infty, 1]) - \bOne(\,(-\infty, \infty)\,)$.
\end{eg}

\begin{proof}
When also $Q\precsim P$, i.e.\ $P$ and~$Q$ have the same normal fan, this is the Brianchon--Gram formula as usually stated \cite[\S1.1]{HaasePolar}.
The general case is obtained by replacing $P$ with $P+\varepsilon Q$, which has the same normal fan as~$Q$,
and then letting $\varepsilon\to0$ from above.
\end{proof}

We get a variant of the Brianchon--Gram formula by polarizing the tangent cones away from~$P$ rather than towards it.
Given a closed simplicial cone $C$, say $C=\bigcap_{i=1}^c H_i$ where the $H_i$ are transverse closed halfspaces,
let $\flip C=\bigcap_{i=1}^c(V^\vee\setminus H_i)$.
The set $\flip C$ is a translate of~$-C^\circ$, so $\bOne(\flip C)\in\IndicatorGroup(V^\vee)$ by Theorem~\ref{McMullenEulerMap}.

\begin{theorem}[Outward Brianchon--Gram formula]\label{BrianchonGramOut}
Let $P\precsim Q$ be polytopes in $V^\vee$. Then
\[\bOne(P) = \sum_{\mbox{\scriptsize $F$ a face of $Q$}}(-1)^{\codim F}\,\bOne(\flip\Tangent_F P)\]
holds in $\IndicatorGroup(V^\vee)$. 
\end{theorem}

\begin{eg}
We continue Example~\ref{LineSegmentEG}; the outward Brianchon-Gram formula in this case says that $\bOne([0,1]) = \bOne((-\infty, \infty)) - \bOne((-\infty, 0)) - \bOne((1,\infty))$. 
\end{eg}

\begin{proof}
Write $\le$ for the facial order on a polyhedron.
For any $F\le Q$, the faces of $\Tangent_FP$ are in bijection with the interval $[F,Q]$ of the face poset of~$Q$, and for $F\le G\le Q$ we have $\Tangent_G\Tangent_F P=\Tangent_GP$.
For $C=\bigcap_{i=1}^c H_i$ a simple cone as above, so that the minimal face of~$C$ has dimension $\dim V^\vee-c$,
we have
\begin{equation}\label{1_C^*}
\bOne(\flip C) = \sum_{G\le C}(-1)^{\codim G}\,\bOne(\Tangent_GC).
\end{equation}
Then Theorem~\ref{BrianchonGram} becomes
\begin{align*}
\bOne(P) &= \sum_{G\le Q}(-1)^{\dim G}\,\bOne(\Tangent_G P)
\\&= \sum_{G\le Q}\left(\sum_{F\le G}(-1)^{\dim F}\right)(-1)^{\dim G}\,\bOne(\Tangent_G P)
\\&= \sum_{F\le Q}(-1)^{\dim F}\sum_{G:F\le G\le Q}(-1)^{\dim G}\,\bOne(\Tangent_G\Tangent_F P)
\\&= \sum_{F\le Q}(-1)^{\dim F}(-1)^{-\dim V^\vee}\bOne(\flip\Tangent_F P).
\end{align*}
The second equality is because every face of $Q$ is contractible, and hence has Euler characteristic $1$. The final equality is  \eqref{1_C^*}.
\end{proof}

\section{Relations in the group of matroid polytopes} \label{DFRelations}

We will now specialize the formulas discussed in Section~\ref{GroupOfPolyhedraSection} to the case of matroid polytopes.
The role of the polytope $Q$ in Section~\ref{GroupOfPolyhedraSection} will be played by the permutahedron. 
The \newword{permutahedron} $\Pi_E$ is the Minkowski sum of the segments $\Hull\{e_i,e_j\}$ over all unordered pairs of distinct $i,j\in E$. 
The vertices of $\Pi_E$ are indexed by total orders $e_1 \le e_2 \le \cdots \le e_n$ on $E$; the vertex corresponding to such a total order is $v(\le) := \sum_i (n-i) e_i$.
The faces of $\Pi_E$ are indexed by chains of sets $\emptyset = S_0 \subsetneq S_1 \subsetneq \cdots \subsetneq S_k = E$; given such a chain $S_{\bullet}$,  
its corresponding face $F(S_{\bullet})$ is the convex hull of those vertices $v(\le)$ where each $S_i$ is an initial segment of $\le$. 
The normal fan of $\Pi_E$ is the fan $\Sigma_E$ from Section~\ref{PermutahedralFanSection}.

We will want the following notations:
We write $2^E$ for the set of all subsets of$E$.
Let $\mathcal P$ be a subset of $2^E$  with a minimum and maximum element $\hat{0}$ and $\hat{1}$.
We write $\Chains(\mathcal P)$ for the set of chains of subsets of~$E$ whose elements are in $\mathcal P$, including $\hat{0}$ and $\hat{1}$.
We number the elements of such a chain as $S_0 \subset S_1 \subset \cdots \subset S_k$ so that $S_0 = \hat{0}$ and $S_k = \hat{1}$. 
We write $|S_{\bullet}| = k$ and call $k$ the \newword{length} of the chain $S_{\bullet}$.

Let $M$ be a rank $r$ matroid on $E$. Then the normal fan of the matroid polytope $\Delta(M)$ is a coarsening of $\Sigma_E$
(\cite[Theorem 19]{Edmonds70}, \cite[Theorem 4.1]{GGMS}),
so $\Delta(M) \precsim \Pi_E$. 
The polytope $\Delta(M)$ will play the role of $P$ in Section~\ref{GroupOfPolyhedraSection}.
Thus, we need to compute $\Tangent_{F(S_{\bullet})} \Delta(M)$.

\begin{prop}
Let $M$ be a matroid and $S_{\bullet}$ a chain of sets as above. Then the inward tangent cone $\Tangent_{F(S_{\bullet})} \Delta(M)$ is cut out by the inequalities
\[  \sum_{j \in S_i} z_j  \leq \rank_M(S_i)    \ \text{for} \ 1 \leq i \leq k-1 \]
and  $\sum_j z_j = r$. 
The flipped tangent cone $\flip \left( \Tangent_{F(S_{\bullet})} \Delta(M) \right)$ is cut out by the inequalities
\[   \sum_{j \in S_i} z_j  > \rank_M(S_i)    \ \text{for} \ 1 \leq i \leq k-1 \]
together with $\sum_j z_j = r$.
\end{prop}

Thus,  Theorems~\ref{BrianchonGram} and~\ref{BrianchonGramOut} imply the formulas
\begin{prop} \label{BrianchonForMatroids}
We have the following relations in the group
$\IndicatorGroup((\RR^E)^{\vee})$:
\begin{align*} \bOne(\Delta(M)) &= \sum_{S_{\bullet} \in \Chains(2^E)} (-1)^{n-|S_{\bullet}|} \ \bOne\Big(\big\{ z \in (\RR^E)^{\vee} : \textstyle{\sum_{j \in S_i}} z_j  \leq \rank_M(S_i),\ \textstyle{\sum_j z_j = r} \big\}\Big)\\
&= \sum_{S_{\bullet} \in \Chains(2^E)} (-1)^{|S_{\bullet}|-1} \ \bOne\Big(\big\{ z \in (\RR^E)^{\vee} : \textstyle{\sum_{j \in S_i} z_j}  > \rank_M(S_i),\ \textstyle{\sum_j z_j} = r \big\}\Big) .  
\end{align*}
\end{prop}

We now intersect the formulas in Proposition~\ref{BrianchonForMatroids} with the polytope $\Delta(r,E)$; we will reencounter the Schubert matroids from Section~\ref{SchubertMatroidsSection}.
To this end, we make the following notations. 
For $M$ a matroid on $E$, and $S_{\bullet}$ a chain of sets in $E$ of the form $\emptyset = S_0 \subsetneq S_1 \subsetneq \cdots \subsetneq S_k = E$, we define $a_M(S_{\bullet})$ to be the vector $(a_0, a_1, a_2, \ldots, a_{k-1}, a_k)$ where $a_i = \rank_M(S_i)$. We define $\Omega_{S_{\bullet}, M}$ to be $\Omega_{S_{\bullet}, a_M(S_{\bullet})}$ for this vector $a_M(S_{\bullet})$.

We will also need a half-open version of the matroid polytope $\Delta(\Omega^{S_{\bullet}, M})$. 
We  define
\[{\halfopenpolytope}^{S_\bullet,a} = \Delta(\Omega^{S_\bullet,a}) \cap
\big\{x\in\mathbb R^E: \textstyle\sum_{j\in S_i}x_j > a_i \big\} \ \text{and} \  {\halfopenpolytope}^{S_\bullet,M} = {\halfopenpolytope}^{S_\bullet,a_M(S_{\bullet})}  .\]

\begin{prop} \label{BoundedBrianchonForMatroids}
Let $M$ be a rank $r$ matroid on $E$. With the above notations, we have the following relations in the group $\IndicatorGroup((\RR^E)^{\vee})$:
\begin{align} \bOne(\Delta(M)) &= \sum_{S_{\bullet} \in \Chains(2^E)} (-1)^{n-|S_{\bullet}|} \ \bOne(\Delta(\Omega_{S_{\bullet}, M}))  
\label{BoundedBrianchonInward}\\
&= \sum_{S_{\bullet} \in \Chains(2^E)} (-1)^{|S_{\bullet}|-1} \  \bOne(\halfopenpolytope^{S_\bullet,M}).
\label{BoundedBrianchonOutward}
\end{align}
\end{prop}

\begin{proof}
These are the formulae in Proposition~\ref{BrianchonForMatroids} with the indicator functions restricted to $\Delta(r, E)$.
\end{proof}

Formula~\eqref{BoundedBrianchonInward} is due to~\cite{DerksenFink}. Formula~\eqref{BoundedBrianchonOutward}, which will be more useful to us, is new to this paper. 

\section{Simplifying the sums to flats} \label{EightFormulaeSection}

Proposition~\ref{BoundedBrianchonForMatroids} gives two ways to express $\bOne(\Delta(M))$ as a sum over $\Chains(2^E)$. 
In this section, we will show how to rewrite  Proposition~\ref{BoundedBrianchonForMatroids} to sum over the smaller collection $\Chains(\mathcal{F}(M))$,
where $\mathcal{F}(M)$ is the set of flats of~$M$.
We recall the convention  from Section~\ref{DFRelations} that chains in $\Chains(\mathcal{P})$ always contain the minimum and maximum elements of $\mathcal{P}$.

Let $\mu_{\mathcal P}$ denote the M\"obius function on a poset $\mathcal P\subseteq 2^E$,
with bottom and top $\hat0,\hat1$.
We extend $\mu$ to chains multiplicatively.
If a chain $S_\bullet$ does not contain both $\hat0$ and $\hat1$, set $\mu_{\mathcal P}(S_\bullet)=0$.
Otherwise set
\[\mu_{\mathcal P}(S_\bullet) := \prod_{i=0}^{|S_\bullet|}\mu_{\mathcal P}(S_i,S_{i+1}).\]

We now give a simplification of each of the two right hand sides in Proposition~\ref{BoundedBrianchonForMatroids}. 
Of the two propositions, Proposition~\ref{DFFlatsOuter} will be the more useful one for us.

\begin{proposition}\label{DFFlatsOuter}
Let $M$ be a loop-free matroid.
We have the following formula for $\bOne(\Delta(M))$, where the sums run over chains of flats. 
\begin{equation}
\bOne(\Delta(M)) = \sum_{(F_0, F_1, \ldots, F_k) \in \Chains(\mathcal{F}(M))} (-1)^{k-1} \bOne(\halfopenpolytope(\Omega^{F_\bullet,M})) . \label{eqn:DFFlatsOuter}
\end{equation}
\end{proposition}

\begin{proof}
We prove this formula from eq.~\ref{BoundedBrianchonOutward} by a \emph{toggling} argument.
We will define a partial involution $\iota$ on~$\Chains(2^E)$ which inserts or removes one element from a chain $S_\bullet$
without changing 
$\halfopenpolytope(\Omega^{S_\bullet,M})$.
Then all terms in \eqref{BoundedBrianchonOutward} 
whose indexing chain is in the domain of $\iota$ will cancel pairwise, and the remaining terms will be \eqref{eqn:DFFlatsOuter}.

Fix a total order $\prec$ on~$E$. We do not use $\prec$ as an argument to define Schubert matroids, only to choose what to toggle.

The domain of $\iota$ will be $\Chains(2^E)\setminus\Chains(\mathcal F(M))$.
Let $S_\bullet$ be a chain in the domain, i.e.\ not consisting entirely of flats of~$M$.
Let $x$ be the $\prec$-maximal element of~$E$ contained in $\cl_M(S_j)\setminus S_j$ for some~$j$,
where $\cl_M$ is the closure operation of~$M$.
Let $k$ be minimal such that $x\in S_k$.  
If $S_k=S_{k-1}\cup\{x\}$ then we define $\iota(S_\bullet)$ to be obtained from $S_\bullet$ by deleting $S_k$.
Otherwise we define $\iota(S_\bullet)$ to be given by inserting the new set $S_{k-1}\cup\{x\}$.
These are visibly inverses.
So we show that $\halfopenpolytope(S_\bullet,M)=\halfopenpolytope(\iota(S_\bullet),M)$. For concreteness, assume that we are in the case $S_k=S_{k-1}\cup\{x\}$, so $\iota(S_{\bullet})$ has one fewer set than $S_{\bullet}$ does.

The polytopes  $\halfopenpolytope(S_\bullet,M)$ and $\halfopenpolytope(\iota(S_\bullet),M)$ are both defined as the intersection of $\Delta(r,E)$ with certain open halfspaces. 
The difference is that the halfspace $\{ \sum_{j \in S_k} z_j > \text{rank}_M(S_k) \}$ is missing in the intersection defining  $\halfopenpolytope(\iota(S_\bullet),M)$. So what we need to show is that the inequality $\sum_{j \in S_k} z_j > \text{rank}_M(S_k)$  already holds everywhere on $\halfopenpolytope(\iota(S_\bullet),M)$, so this condition is redundant.
Indeed, on $\halfopenpolytope(\iota(S_\bullet),M)$, we have $\sum_{j \in S_{k-1}} z_j  > \text{rank}_M(S_{k-1})$. But, since $S_k = S_{k-1} \cup \{ x \}$ and $x \in \cl_M(S_{k-1})$, we have $\text{rank}_M(S_{k-1}) = \text{rank}_M(S_k)$, so $\sum_{j \in S_{k-1}} z_j  > \text{rank}_M(S_{k})$ on $\halfopenpolytope(\iota(S_\bullet),M)$. We have $\sum_{j \in S_{k}} z_j  \geq 
\sum_{j \in S_{k-1}} z_j$ on $\Delta(r,E)$, so we see that the inequality  $\sum_{j \in S_k} z_j > \text{rank}_M(S_k)$  already holds on $\halfopenpolytope(\iota(S_\bullet),M)$, as we want.
\end{proof}

The previous proposition used the half-open polytope from eq.~\ref{BoundedBrianchonOutward}.
We now give a variant using the closed polytopes in eq.~\ref{BoundedBrianchonInward}; this formula is more complicated (because of the M\"obius functions) and less useful.

\begin{proposition}\label{DFFlatsInner}
Let $M$ be a loop-free matroid.
We have the following formula for $\bOne(\Delta(M))$, where the sums run over chains of flats. 
\begin{equation}
 \bOne(\Delta(M)) = \sum_{(F_0, F_1, \ldots, F_k) \in \Chains(\mathcal{F}(M))} (-1)^k\,\mu_{\mathcal F(M)}(F_\bullet)\, \bOne(\Delta(\Omega_{F_\bullet,M})) . \label{eqn:DFFlatsInner}
 \end{equation}
\end{proposition}

\begin{remark}
Hampe \cite[Theorem~3.12]{Hampe} is a variant of Proposition~\ref{DFFlatsInner}, where we sum only over cyclic flats.
\end{remark}

\begin{proof}
As in the previous proof, we use a toggling argument. This time, the domain of~$\iota$ will consist of those chains 
which do not contain the closure $\cl_M(S_i)$ of at least one of their members $S_i$.
Suppose that $S_{\bullet}$ is a chain which does not contain $\cl_M(S_i)$ for some $S_i$. 
Let $j$ be the index where $\text{rank}_M(S_i) = \text{rank}_M(S_j) < \text{rank}_M(S_{j+1})$. Then $\cl_M(S_i) = \cl_M(S_j)$, and there must be some $x \in \cl_M(S_j)$ which is not in $S_j$. Thus, there is some index $j$ and some element $x \in E$ such that $x \in \cl_M(S_j)\setminus S_j$ and $\rank_M(S_{j+1})>\rank_M(S_j)$.
Let $x$ be the $\prec$-maximal such element of $x$ such that such a $j$ exists and let $k$ be minimal such that $x \in S_k$.

If $S_k=S_{k-1}\cup\{x\}$, then we define $\iota(S_\bullet)$ by deleting $S_{k-1}$
(not $S_k$ as before).
Otherwise we take $\iota(S_\bullet)$ to be given by inserting $S_k\setminus\{x\}$.
Observe that, in the former case, $\rank_M(S_k)=\rank_M(S_{k-1})$;
therefore the set we delete is not the $S_j$ that witnessed the choice of~$x$,
and the elements that are candidates to be~$x$ are unchanged by these insertions or removals.
With this, the insertion and removal are seen to be inverses of each other.

The check that $\Omega_{S_\bullet, M}=\Omega_{\iota(S_\bullet), M}$ runs as it did in Proposition~\ref{DFFlatsOuter}.

Given a chain of sets  $S_{\bullet}$, let $F_{\bullet}$ be the chain of flats $\cl_M(S_0) \subseteq \cl_M(S_1) \subseteq \cdots \subseteq \cl_M(S_k)$, with duplicate entries deleted. 
We will write $F_{\bullet} \subseteq S_{\bullet}$ as shorthand to mean $\{ \cl_M(S_0), \cl_M(S_1), \cdots, \cl_M(S_k) \} \subseteq \{ S_0, S_1, \ldots, S_k \}$.
Our toggling argument shows that
\[ \bOne(\Delta(M))  = \sum_{\substack{ S_{\bullet} \in \Chains(2^E) \\ F_{\bullet} \subseteq S_{\bullet} }} (-1)^{n-|S_{\bullet}|}\, \bOne(\Delta(\Omega(S_{\bullet}, M))) . \]

We claim that, whenever $F_{\bullet} \subseteq S_{\bullet}$, then we have  $\Omega(S_{\bullet}, M) = \Omega(F_{\bullet}, M)$. 
Recall that $B$ is a basis of  $\Omega(S_{\bullet}, M)$ if and only if $|B \cap S_j| \leq \text{rank}_M(S_j)$ for all $j$. We need to show that the conditions corresponding to those $S_j$ which are not flats are redundant.
Indeed, suppose that $S_j$ is not a flat, and let $k$ be the index such that $\cl_M(S_j) = S_k$. Then $\text{rank}_M(S_j) =  \text{rank}_M(S_k)$. However, since $S_j \subseteq S_k$, we have $|B \cap S_j| \leq |B \cap S_k|$, so the inequality coming from $S_k$ is more restrictive. This concludes the verification that  $\Omega(S_{\bullet}, M) = \Omega(F_{\bullet}, M)$. 
Thus, we can rewrite our sum as
\[\sum_{\substack{ S_{\bullet} \in \Chains(2^E) \\ F_{\bullet} \subseteq S_{\bullet} }} (-1)^{n-|S_{\bullet}|}\, \bOne(\Delta(\Omega(F_{\bullet}, M))) . \]

The coefficent of  $\bOne(\Delta(\Omega(F_{\bullet}, M)))$ in the above sum is the sum of $(-1)^{n-|S_{\bullet}|}$ where the sum is over all chains of sets $S_{\bullet}$ such that $F_{\bullet} \subseteq S_{\bullet}$ and $\cl_M(S_j) \in F_{\bullet}$ for each set $S_j$. 
To obtain such a chain of sets $S_{\bullet}$, one starts with the chain of flats $F_{\bullet}$ and then, between $F_{i-1}$ and $F_i$, inserts a chain of sets $S$ such that $F_{i-1} \subsetneq S \subsetneq F_i$ and $\cl_M(S) = F_i$. In other words, $S \setminus F_{i-1}$ should be a proper spanning set in the matroid $M|{F_i} / F_{i-1}$. 
The insertions between each $F_{i-1}$ and $F_i$ can be chosen independently, so we get that the coefficient of $\bOne(\Delta(\Omega(F_{\bullet}, M)))$ is the product over the contributions from each $(F_{i-1}, F_i)$ pair. 
Splitting the factors of $-1$ according to
\[n-|S_\bullet| = \sum_{i=1}^{|F_\bullet|}|F_i|-|F_{i-1}|-\#\{T\in S_\bullet:F_{i-1}\subsetneq T\subseteq F_i\},\]
we want to compute
\[ \sum_{\substack{F_{i-1} \subsetneq T_1 \subsetneq T_2 \subsetneq \cdots T_q \subsetneq F_i \\ \text{Span}(T_j) = F_i }} (-1)^{|F_i|-|F_{i-1}|-(q+1)} . \]
Let $T_{q+1}$ refer to $F_i$.
If $\text{Span}(T_1)=F_i$ then $T_1\subseteq T_j\subseteq F_i$ implies $\text{Span}(T_j)=F_i$ for all~$j$. So we may rewrite the last sum as
\begin{align*}
&\mathrel{\phantom=}
\sum_{\substack{F_{i-1} \subsetneq T_1 \subseteq F_i \\ \text{Span}(T_1) = F_i }} (-1)^{|F_i|-|F_{i-1}|-1} 
  \sum_{T_1 \subsetneq T_2 \subsetneq \cdots T_q \subsetneq F_i} (-1)^{q} \\
&= \sum_{\substack{F_{i-1} \subsetneq T_1 \subseteq F_i \\ \text{Span}(T_1) = F_i }} (-1)^{|F_i|-|F_{i-1}|-1}
  \cdot (-1)^{|F_i|-|T_1|} \\
&= -\sum_{\substack{F_{i-1} \subsetneq T_1 \subseteq F_i \\ \text{Span}(T_1) = F_i }} (-1)^{|T_1|-|F_{i-1}|}.
\end{align*}
By Rota's crosscut theorem \cite[Theorem 3.1.9]{KungRotaYan}, 
this is $-\mu(M|{F_i} / F_{i-1})$. Putting all the pieces together, we get formula~\eqref{eqn:DFFlatsInner}.
\end{proof}

For any valuative matroid invariant $v$, we can apply $v$ to both sides of the formulas in Proposition~\ref{DFFlatsOuter} or~\ref{DFFlatsInner} to obtain a formula for $v$ on an arbitrary matroid in terms of the values of $v$ on Schubert matroids. We will now discuss what happens when we apply this strategy to the $\omega$ invariant.

\section{Ferroni paths}

We want to understand the values the $\omega$ invariant takes on Schubert matroids. A formula for this was given by Ferroni~\cite{Ferroni}. 
We will give a different formulation of Ferroni's result; in the proof of Theorem~\ref{FerroniClosed} we explain how it relates to Ferroni's original formulation.

Let $\emptyset = S_0 \subset S_1 \subset \cdots \subset S_k = E$ be a flag of sets and let $(a_0, a_1, \ldots, a_k)$ be a sequence of nonnegative integers with $a_0 = 0$, $a_k = r$ and $a_i \leq a_{i+1} \leq a_i + |S_{i+1} \setminus S_i|$.
Define $\verts(S_{\bullet}, a_{\bullet})$ to be the set of lattice points $\{ (|S_j|-a_j, a_j) : 1 \leq j \leq k-1 \}$. 
Fix a matroid structure on $E$.
For any subset $S$ of $E$, we define $p_M(S)$ to be the lattice point $(|S| - \rank(S), \rank(S))$, so, for any chain $S_{\bullet} = (S_0, S_1, \ldots, S_k)$, we have $$\verts(S_{\bullet}, a(S_{\bullet})) = (p_M(S_1), p_M(S_2), \ldots, p_M(S_{k-1})).$$
If the matroid is clear from context, we drop the subscript $M$ in the notation. 

A \newword{Ferroni path} is a lattice path in~$\mathbb Z^2+(\frac12, \frac12)$ whose steps are $(1,1)$ or~$(1,0)$.
Ferroni paths are the case of the Delannoy paths of~\cite{Ferroni} without $(0,1)$ steps.
We deviate from \cite{Ferroni} in taking the vertices of Delannoy paths to be half-integer points, while taking the points of $\verts(S_{\bullet}, a_{\bullet})$ to be integer points.
We do this in the interests of symmetry and the ease of stating the conventions distinguishing the next two theorems.

\begin{theorem}[{\cite[Theorem 3.4]{Ferroni}}]\label{FerroniClosed}
For $S_{\bullet}$ and $a_{\bullet}$ as above, $\omega(\Omega_{\leq, A})$  equals the number of Ferroni paths from $(\frac12, \frac12)$ to~$(n-r-\frac12,r-\frac12)$
that lie strictly below the points $\verts(S_{\bullet}, a_{\bullet})$. 
\end{theorem}
By Corollary~\ref{OmegaValCoval}, this formula also gives the value of $\omega^{\circ}(\Omega_{\leq, A})$. 

\begin{proof}
Ferroni works with $\Omega_{\leq, A}$ rather than $\Omega_{S_{\bullet}, a_{\bullet}}$, and he counts Ferroni paths lying below a certain piecewise linear path from $(0,0)$ to $(n-r, r)$ rather than lying below a set of vertices. 
We explain how to translate both issues.

We recall the recipe that, given $(S_{\bullet}, a_{\bullet})$, computes a pair $(\leq, A)$ such that $\Omega_{S_{\bullet}, a_{\bullet}} = \Omega_{\leq, A}$: We choose any total order $\leq$ for which the $S_i$ are initial segments.
We take $A_i$ to consist of the first $a_i - a_{i-1}$ elements of $S_i \setminus S_{i-1}$ and let $A = \bigsqcup_i A_i$. 

Ferroni defines a path $\path(A)$ from $(0,0)$ to $(n-r, r)$ by having the $i$-th step of $\path(A)$ travel in direction $(0,1)$ if $i \in A$ and in direction $(1,0)$ otherwise; his result is that $\omega(\Omega_{\leq, A})$ is the number of Ferroni paths strictly below $\path(A)$.  The points $\verts(S_{\bullet}, a_{\bullet})$ are the lower right corners of path $\path(A)$;  a Ferroni path $F$ thus lies below $\path(A)$ if and only if it lies below $\verts(S_{\bullet}, a_{\bullet})$. 
\end{proof}

\begin{eg}\label{FerroniEx}
We consider the rank $4$ Schubert matroid on the ground set $[10]$ described as $\Omega_{(\emptyset, [2], [7], [10]), (0,1,3,4)} = \Omega_{\leq, \{1,3,4,8\}}$, where $\leq$ is the standard order on $[10]$. Then \[\verts((\emptyset, [2], [7], [10]), (0,1,3,4))  = \{  (1,1), (4,3) \}.\]
Figure~\ref{FerroniExFig} shows these vertices as dots. 
The Ferroni paths travel along the dashed lines. 
Of the ten Ferroni paths from $(\frac12, \frac12)$ to $(5\frac12,3\frac12)$, 
just three are strictly below this Young path, 
namely the three that pass through $(1\frac12,\frac12)$ and $(4\frac12,2\frac12)$.
This shows that $\omega(\Omega_{\leq, \{1,3,4,8\}}) = 3$. 
\end{eg}
\begin{figure}[htb]
\includegraphics{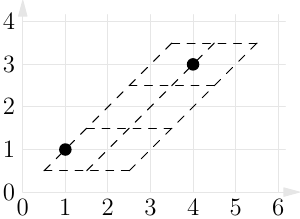}
\caption{The Ferroni paths  and lattice points $\verts(S_{\bullet}, a_{\bullet})$ from  Example~\ref{FerroniEx}.}\label{FerroniExFig}
\end{figure}

\begin{cor} \label{FerroniInward}
Let $M$ be any rank $r$ matroid on $E$. Then
\begin{align*}
\omega^{\circ}(M) =& \sum_{S_{\bullet} \in \Chains(\mathcal{S}(M))} \!\!\! \!\!\! \!\!\! (-1)^{|S_{\bullet}|-1}\, {}^\#\{ \textup{Ferroni paths strictly below $\verts(S_{\bullet}, a_M(S_{\bullet}))$} \} \\
=& \sum_{F_{\bullet} \in \Chains(\mathcal{F}(M))} \!\!\! \!\!\! \!\!\! (-1)^{|F_{\bullet}|-1}\,\mu_{\mathcal F(M)}(F_\bullet) {}^\#\{ \textup{Ferroni paths strictly below $\verts(F_{\bullet}, a_M(F_{\bullet}))$} \} .
\end{align*}
\end{cor}

\begin{proof}
Apply the covaluation $\omega^{\circ}$ to both sides of equations~\eqref{BoundedBrianchonInward} and~\eqref{eqn:DFFlatsInner} respectively. By Corollary~\ref{OmegaValCoval}, we don't have to distinguish between $\omega$ and $\omega^{\circ}$ on the right hand side, since all matroids on that side are Schubert matroids.
\end{proof}

To give the corresponding formulas when we apply $\omega^{\circ}$ to equations~\eqref{BoundedBrianchonOutward} and~\eqref{eqn:DFFlatsOuter}, 
we need to work with the half open polytopes 
$\halfopenpolytope(\Omega^{F_\bullet,M})$.
Since $\omega^{\circ}$ is valuative, it extends to a linear functional on the subspace of $\IndicatorGroup(V^\vee)$ spanned by indicator functions of matroid polytopes; we denote this extension by $\omega^{\circ}$ as well.
\begin{theorem} \label{FerroniOpen} 
Let $S_{\bullet}$ and $a_{\bullet}$ be as above and let  $\halfopenpolytope^{S_\bullet,a}$ be the half open polytope introduced in Section~\ref{DFRelations}. Then $\omega^{\circ}(\bOne(\halfopenpolytope^{S_\bullet,a}))$ is the number of Ferroni paths lying weakly above $\verts(S_{\bullet}, a_{\bullet})$. 
\end{theorem}
By lying weakly above, we mean that the Ferroni path may pass through the points in $\verts(S_{\bullet}, a_{\bullet})$ but may not pass strictly below these points.  

Comparing the last two results explains our use of just vertices instead of a Young path to bound the Ferroni paths:
both invoke $\verts(S_{\bullet}, a_{\bullet})$, but they cannot be stated in terms of the same bounding Young path.
To state Corollary~\ref{FerroniInward} using a path we would have needed its lower right corners to occur in the positions  $\verts(S_{\bullet}, a_{\bullet})$, 
whereas a path that suits Theorem~\ref{FerroniOpen} would instead need
to have upper left corners at $\verts(S_{\bullet}, a_{\bullet})$.

\begin{proof}
Let $S_{\bullet}$ have length $k$, so $S_0 = \emptyset$ and $S_k = [n]$.
Recall that 
\[ \halfopenpolytope^{S_{\bullet}, a_{\bullet}} = \Delta(r, E) \cap \bigcap_{i=1}^{k-1} \left\{ z : \sum_{e \in S_i} z_e > a_e  \right\} . \] 
For each subset $J = \{ j_1 < j_2 < \cdots < j_c \}$ of $[k-1]$, let $S|_J$ be the flag of sets $S_0 \subset S_{j_1} \subset S_{j_2} \subset \cdots \subset S_{j_c} \subset S_k$ and let $a|_J$ be the sequence $(a_0, a_{j_1}, a_{j_2}, \ldots, a_{j_c}, a_k)$. Then
\[ \bOne( \halfopenpolytope^{S_{\bullet}, a_{\bullet}} ) = \sum_{J \subseteq [k-1]} (-1)^{|J|} \bOne(\Omega_{S|_J, a|_J}) .\]
Applying $\omega^{\circ}$ to both sides of this equation, we need to count Ferroni paths lying strictly below $\verts(S|_J, a|_J)$, weighted by $(-1)^{|J|}$. 

Given a Ferroni path $F$, let $K$ be the set of indices $i$ for which $(|S|_i - a_i, a_i)$ lies weakly below $F$. So $F$ contributes to the count for those $J$ where $J \cap K = \emptyset$, and the total contribution of $F$ is $\sum_{J \subseteq [k-1] \setminus K} (-1)^{|J|}$. This is $1$ if $K = [k-1]$ and $0$ otherwise. So we only obtain a contribution if all the points of $\verts(S_{\bullet}, a_{\bullet})$ are weakly below $F$, or in other words $F$ is weakly above $\verts(S_{\bullet}, a_{\bullet})$, and this contribution is $1$.
\end{proof}

We now can apply  $\omega^{\circ}$ to equations~\eqref{BoundedBrianchonOutward} and~\eqref{eqn:DFFlatsOuter} to obtain:
\begin{cor} \label{FerroniOutward}
Let $M$ be any rank $r$ matroid on $E$. Then
\begin{align*}
 \omega^{\circ}(M) =& \sum_{S_{\bullet} \in \Chains(\mathcal{S}(M))} \!\!\! \!\!\! \!\!\! (-1)^{|S_{\bullet}|-1}\, {}^\#\{ \textup{Ferroni paths weakly above $\verts(S_{\bullet}, a_M(S_{\bullet}))$} \} \\
 = &\sum_{F_{\bullet} \in \Chains(\mathcal{F}(M))} \!\!\! \!\!\! \!\!\! (-1)^{|F_{\bullet}|-1}\, {}^\#\{ \textup{Ferroni paths weakly above $\verts(F_{\bullet}, a_M(F_{\bullet}))$} \} .
 \end{align*}
\end{cor}
 
\section{Crowded sets and cancellation of summands}
Define the \newword{crowding} of a set $S$ in a matroid $M$ as
\[\stress_M(S) = |S|-2\rank_M(S).\]
We will omit the subscript $M$ when it is clear from context.

Writing $r=\rank M$, our standing assumption $n\ge 2r$ implies that $\stress_M(E)=n-2r\ge0$.
We say that $S$ is \newword{crowded} (in~$M$) if $\stress_M(S) \geq 0$ and \newword{uncrowded} otherwise.
Reader beware: in our terminology, a set of crowding zero \textbf{is} crowded.
Write $\StressedSets(M)$ for the poset of crowded sets in~$M$ and $\StressedFlats(M)$ for the poset of crowded flats.

\begin{lemma} \label{StressSupermodular}
The function $\stress$ is supermodular, meaning that, for any subsets $I$ and $J$ of $E$, we have
\[ \stress(I \cup J) + \stress(I \cap J) \geq \stress(I) + \stress(J) .\]
\end{lemma}

\begin{proof}
We want to show that
\[ |I \cup J| - 2 \rank(I \cup J) + |I \cap J| - 2 \rank(I \cap J) \geq |I| - 2 \rank(I) + |J| - 2 \rank(J) . \]
By inclusion-exclusion, the cardinality terms cancel, so the inequality simplifies to
\[ \rank(I) + \rank(J) \geq \rank(I \cup J) + \rank(I \cap J) , \]
the standard fact that $\rank$ is submodular \cite[Lemma 1.3.1]{Oxley}.
\end{proof}

We will be examining sets of high crowding in~$M$. 
Often we can restrict our attention to flats rather than to all sets:
\begin{lemma} \label{StressClosure}
Let $S$ be a subset of $E$ and let $F$ be the closure of $S$ in the matroid $M$. Then $\stress(F) \geq \stress(S)$, with strict inequality if $F \neq S$.
\end{lemma}

\begin{proof}
We have $\rank(F) = \rank(S)$ and $|F| \geq |S|$.
\end{proof}

The importance of crowding is that uncrowded sets and sets of high crowding compared to~$E$ impose very strong conditions on Ferroni paths,
depicted in Figure~\ref{StressAndFerroniFig} and stated in the next proposition.
\begin{prop} \label{StressAndFerroni}
Let $M$ be a rank $r$ matroid on $E$, with $|E|=n$, and let $S \subseteq E$. 
If $S$ is uncrowded, then all Ferroni paths from $(\frac12, \frac12)$ to $(n-r-\frac12, r-\frac12)$ pass strictly below $p_M(S)$. 
If $\stress(S)=n-2r$, then all Ferroni paths pass weakly above $p_M(S)$, and if $\stress(S)>n-2r$, then all such Ferroni paths pass strictly above $p_M(S)$.
\end{prop}

\begin{proof}
We use $(x,y)$ to denote coordinates in the plane where we draw Ferroni paths.

The highest possible Ferroni path is the one which consists of $r-1$ segments in direction $(1,1)$, followed by $n-2r$ segments in direction $(1,0)$. The first part of this path travels along the line $x=y$. If $S$ is uncrowded, then $p_M(S)$ lies strictly above this line, and hence strictly above any Ferroni path.

The lowest possible Ferroni path is the one which consists of  $n-2r$ segments in direction $(1,0)$ followed by $r-1$ segments in direction $(1,1)$. The second part of this path travels along the line $x = y+(n-2r)$. If $\stress(S)=n-2r$, then $|S|-\rank_M(S) = \rank_M(S) + (n-2r)$, so $p_M(S)$ lies on this line and hence weakly below any Ferroni path; if $\stress(S)>n-2r$ then, similarly, $p_M(S)$ is strictly below any Ferroni path.
\end{proof}

\begin{figure}[htb]
\includegraphics[scale=0.75]{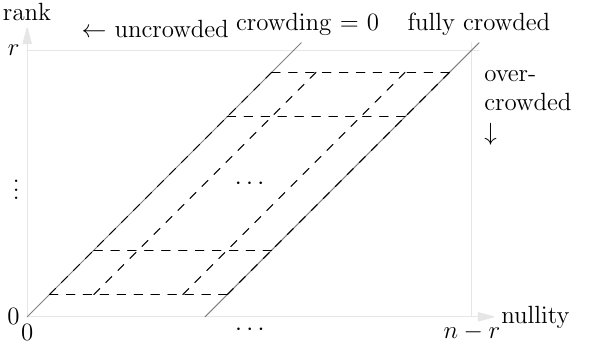}
\caption{Illustration of Proposition~\ref{StressAndFerroni}.
The Ferroni paths for a rank~$r$ matroid on $n$ elements make up the grid of dashed lines.
The crowding coordinate in the picture is $x-y$: 
that is, sets $S$ with $p_M(S)=(x,y)$ have $\stress(S)=x-y$.
The lines on which sets of crowding $0$ and $n-2r$ lie have been drawn in.
}\label{StressAndFerroniFig}
\end{figure}

We now extract consequences of Proposition~\ref{StressAndFerroni}.

\begin{prop} \label{stressedChains}
Let $M$ be a rank $r$ matroid on $E$, with $|E|=n$. 
In the sums in Corollary~\ref{FerroniOutward}, 
if we sum only over chains of crowded sets, respectively crowded flats,
we still recover $\omega^{\circ}(M)$.
That is to say,
\begin{align*}
 \omega^{\circ}(M) =& \sum_{S_{\bullet} \in \Chains(\StressedSets(M))} \!\!\! \!\!\! \!\!\! (-1)^{|S_{\bullet}|-1}\, {}^\#\{ \textup{Ferroni paths weakly above $\verts(S_{\bullet}, a_M(S_{\bullet}))$} \} \\
 = &\sum_{F_{\bullet} \in \Chains(\StressedFlats(M))} \!\!\! \!\!\! \!\!\! (-1)^{|F_{\bullet}|-1}\, {}^\#\{ \textup{Ferroni paths weakly above $\verts(F_{\bullet}, a_M(F_{\bullet}))$} \} 
 \end{align*}
\end{prop}

\begin{proof}
We claim that any chain which contains a  proper, nonempty uncrowded set contributes $0$ to the sum. Indeed, suppose that $S$ is a proper, nonempty uncrowded set. Then, by Proposition~\ref{StressAndFerroni}, there are no Ferroni paths passing weakly above $p_M(S)$, so $S_{\bullet}$ contributes $0$ to the sum in  Corollary~\ref{FerroniOutward}. 
\end{proof}

\begin{prop} \label{NoStress}
Let $M$ be a rank $r$ matroid on $E$, with $|E|=n$, and suppose that $M$ has no non-empty proper crowded flats. Then $\omega(M) = \binom{n-r-1}{r-1}$.
\end{prop}

\begin{proof}
By Proposition~\ref{stressedChains}, we only need to compute the contribution of the trivial chain $(\emptyset, E)$. 
The contribution of the trivial chain is the number of Ferroni paths from $(\frac12, \frac12)$ to $(n-r-\frac12, r-\frac12)$. Each such path has $(n-2r)$ horizontal steps and $(r-1)$ diagonal steps, and those steps may be ordered in any way, so there are $\binom{n-r-1}{r-1}$ such paths.
\end{proof}

\begin{remark}
We could replace ``flats" with ``sets" in Proposition~\ref{NoStress}, but then our result would only be useful in the case $n=2r$, 
since $E \setminus \{ i \}$ is crowded for all $i \in E$ when $n>2r$.
\end{remark}

We will now prove a sequence of results giving conditions under which we can combine different terms in~Proposition~\ref{stressedChains}. For this purpose, we need to know when two chains of sets have the same Ferroni paths lying above them.
Let $S_{\bullet}$ be a chain of sets including $\emptyset$ and $E$. We define the \newword{crowd hull} of $S_{\bullet}$ to be
\[ \StressHull(S_{\bullet}) = \{ S_i \in S_{\bullet} : \Stress(S_j) > \Stress(S_i)  \ \textup{for all}\  j > i \} . \]

\begin{prop} \label{StressHull}
Let $S_{\bullet}$ be a chain of sets including $\emptyset$  and $E$. The set of Ferroni paths lying weakly above $\verts(S_{\bullet})$ is the same as the set of Ferroni paths lying weakly above $\verts(\StressHull(S_{\bullet}))$. 
\end{prop}

\begin{proof}
The points in $\verts(\StressHull(S_{\bullet}))$ are a subset of the points in  $\verts(S_{\bullet})$, so what we need to confirm is that, if $S_i$ is a set in $S_{\bullet}$ which is not in $\StressHull(S_{\bullet})$, then the condition that paths lie above $p_M(S_i)$ is redundant. Let $S_i$ be such a set. Then there is an index $j > i$ with $\Stress(S_j) \leq \Stress(S_i)$. Choose the greatest such $j$. Then $S_j \in \StressHull(S_{\bullet})$. Any path which lies weakly above $p_M(S_j)$ must also lie above $p_M(S_i)$. 
\end{proof}

\begin{remark} \label{BetterHullRemark}
$\StressHull(S_{\bullet})$ is \textbf{not} the smallest subsequence $T_{\bullet}$ of $S_{\bullet}$ such that the Ferroni paths lying above $\verts(T_{\bullet})$ are the same as those lying above $\verts(S_{\bullet})$. The smallest such subsequence is, instead, given by
\begin{equation}
 \{ S_i \in S_{\bullet} : \Stress(S_j)  > \Stress(S_i) \ \textup{for}\  j > i  \ \textup{and} \rank(S_j) <  \rank(S_i) \ \textup{for} \ j<i \}. \label{BetterHullEqn}
 \end{equation}
In other words, if we have two sets $S_i \subset S_j$ in $S_{\bullet}$ with the same rank (and hence the same closure), we should eliminate the former one. 
However, Equation~\ref{BetterHullEqn} appears harder to work with than $\StressHull(S_{\bullet})$. 
We note that, if $S_{\bullet}$ is a chain of flats, then the condition on rank is automatically true and Equation~\ref{BetterHullEqn} is the same as $\StressHull(S_{\bullet})$.
\end{remark}

So, in Proposition~\ref{stressedChains}, we can group together all chains that have the same crowd hull. 
In many cases, the grouped sum turns out to be $0$. 
We will now introduce the notation to describe when this occurs.

Given sets $T \subset S$ in $E$, we say that $T$ is a \newword{summand} of $S$ if $\rank(S) = \rank(T) + \rank(S \setminus T)$. We rephrase this in two equivalent ways: $T$ is a summand of $S$ if and only if  $M|S \cong M|T \oplus M|(S \setminus T)$, which explains the terminology. Also, $T$ is a summand of $S$ if and only if $T$ is a union of connected components of $M|S$.

Given sets $T \subset S$ in $E$, we say that $T$ is \newword{overcrowded in $S$} if either
\begin{enumerate}
\item $\stress(T) > \stress(S)$ or
\item $\stress(T) = \stress(S)$ and $T$ is \textbf{not} a summand of $S$.
\end{enumerate}
We will say that $S$ is a \newword{crowding record} if it has no subset that is overcrowded in it.
We prove some basic lemmas about this concept: 

\begin{lemma} \label{RecordStressorStressed}
If $S$ is a crowding record, then $\stress(S) \geq 0$.
\end{lemma}

\begin{proof}
We have $\stress(S) \geq \stress(\emptyset) \geq 0$.
\end{proof}

\begin{lemma} \label{HighlyStressedComplement}
Let $T \subset S$ with $T$ overcrowded in $S$.
Then $S \setminus T$ is uncrowded, meaning that $\stress(S \setminus T) < 0$.
\end{lemma}

\begin{proof}
By the supermodularity of crowding, we have
\[ \stress(T) + \stress(S \setminus T) \leq \stress(S) + \stress(\emptyset) = \stress(S). \]
So $\stress(S \setminus T) \leq \stress(S)-\stress(T) \leq 0$. 
We need to show that we do not have equality, so suppose that we have equality. Then we must have $\stress(T) = \stress(S)$, and we must have $\rank_M(S) = \rank_M(T) + \rank_M(S \setminus T)$. 
\end{proof}

\begin{lemma} \label{RecordSummand}
Let $S$ be a crowding record and let $X$ be a summand of $S$ (meaning that $M|S \cong M|X \oplus M|(S \setminus X)$). Then $X$ is a crowding record. In particular, $\stress(X) \geq 0$.
\end{lemma}

\begin{proof}
Let $Y = S \setminus X$, so that $\rank(S) = \rank(X) + \rank(Y)$. If $T$ is overcrowded in $X$, then $T \sqcup Y$ is overcrowded in $S$, a contradiction, so we see that no set is overcrowded in~$X$. In particular, $\stress(X) \geq 0$.
\end{proof}

We now start to take advantage of the concept of crowding records,
to give strengthenings of Proposition~\ref{stressedChains}.

\begin{proposition}\label{RecordStressorOnly}
In the first sum of Proposition~\ref{stressedChains}, we can instead sum only over chains of sets all of whose elements are crowding records.
\end{proposition}

We emphasize that Proposition~\ref{RecordStressorOnly} includes the condition that $E$, which is the last element of every chain, must be a crowding record. In other words, if $E$ is not a crowding record, then the sum is empty. So we deduce:

\begin{cor}\label{overStressed}
If $M$ has an overcrowded set in~$E$ --- 
that is, a set of crowding exceeding $n-2r$, or a set of crowding $n-2r$ which is not a union of connected components of~$M$ --- 
then $\omega(M)=0$.
\end{cor}

\begin{proof}[Proof of Proposition~\ref{RecordStressorOnly}]
For each set $S$ which is not a crowding record, choose once and for all an overcrowded subset $T_S$ of $S$, and choose this $T_S$ to be of highest crowding among all sets which are overcrowded in $S$.
Set $U_S = S \setminus T_S$, so $U_S$ is uncrowded by Lemma~\ref{HighlyStressedComplement}.

We will now define an involution, the \newword{toggle}, on the collection of chains of crowded sets, so that every chain that contains a set which is not a crowding record toggles to another such chain, and these two will cancel in the summation.

Let $S_{\bullet}$ be a chain of crowded sets.
If all the sets of $S_{\bullet}$ are crowding records, then we define $S_{\bullet}$ to be its own toggle. 
Otherwise $S$ contains a set which is not a crowding record.
Let $m$ be the smallest index for which $S_m$ is  which is not a crowding record.
Abbreviate $T := T_{S_m}$ and $U := S_m \setminus T$.
For each $1 \leq j \leq m$, we have $S_j \setminus S_{j-1} \subseteq S_m = T \sqcup U$.
We create a string of $m$ symbols $(\alpha_1, \alpha_2, \ldots, \alpha_m)$ in the alphabet $\{ \tt, \tu, \tv \}$ where the symbol $\alpha_j$ is 
\begin{enumerate}
\item[$\tt$] if $S_j \setminus S_{j-1} \subseteq T$.
\item[$\tu$] if $S_j \setminus S_{j-1} \subseteq U$.
\item[$\tv$] if $S_j \setminus S_{j-1}$ contains elements of both $T$ and  $U$. 
\end{enumerate}
The  reader can think of $\tv$ as standing for ``variegated".

Since $U$ is uncrowded, it is not one of the sets in $S_{\bullet}$, so $(\alpha_1, \alpha_2, \ldots, \alpha_m)$  cannot be of the form $\tu^a \tt^b$. Thus, there must either be some index $i$ where $\alpha_i=\tv$, or there must be some index $i$ where $(\alpha_i, \alpha_{i+1}) = (\tt,\tu)$. Let $i$ be the first index where one of these two things happens. 
In the first case, we define the toggle of~$S_{\bullet}$ to be the chain obtained by deleting $S_i$ from the chain of sets; 
in the second case, we define the toggle of~$S_{\bullet}$ to be obtained by inserting $S_{i-1} \cup T  \cap S_i$ in between $S_{i-1}$ and $S_i$. (We note that, since $S_{i-1} \subset S_i$, we have $\left( S_{i-1} \cup T  \right) \cap S_i = S_{i-1} \cup \left(  T  \cap S_i \right)$.)

The two non-identity cases of the toggle are clearly mutually inverse. We must verify that they preserve
\begin{enumerate}
\item The crowd hull of the chain of sets.
\item The fact that all the sets in the chain are crowded.
\item The fact that $S_m$ is   the largest index for which $S_m$ is not a crowding record. 
\end{enumerate}
The last claim is clear, since we modify the chain before $S_m$. In order to study the other two, we must study the crowding of  the set which is inserted or deleted.

Assume without loss of generality that we are in the case where $\alpha_i=\tv$. Let $R = S_{i-1} \cup T \cap S_i$ and let $S'_{\bullet}$ be the chain where we have inserted $R$ into $S_{\bullet}$. 

By our choice of $i$, the string leading up to $\alpha_i$ must be of the form $\tu^j \tt^{i-j-1} \tv$ (possibly with $j$ and/or $i-j-1$ equal to~0). Thus, $S_j \cup T = S_{i-1} \cup T$ and so $R= S_j \cup T \cap S_i$.

By the supermodularity of crowding (Lemma~\ref{StressSupermodular}), we have
\[ \stress(R) \geq \stress(S_j) + \stress(T \cap S_i) - \stress(S_j \cap T \cap S_i) .\]
Since $\alpha_1=\alpha_2=\cdots = \alpha_j = \tu$, we have $S_j \cap T = \emptyset$ and thus $\stress(S_j \cap T \cap S_i)  = \stress(\emptyset) = 0$. 
Also, since $S_j$ is crowded, we have $\stress(S_j) \geq 0$. So we have shown that $\stress(R) \geq \stress(T \cap S_i)$. 

We now apply supermodularity again:
\[  \stress(T \cap S_i) \geq \stress(T) + \stress(S_i) - \stress(T \cup S_i) \geq \stress(S_i) \]
where the second inequality is by the choice of $T$. Putting it all together, we have
\[ \stress(R) \geq \stress(S_i). \]

We use this to address the two issues left open above. Since $\stress(R) \geq \stress(S_i)$ and $R$ is inserted before $S_i$, the set $R$ does not appear in the crowd hull of $S'_{\bullet}$, so $\StressHull(S_{\bullet}) = \StressHull(S'_{\bullet})$. Also, since $S_i$ is crowded, we have $\stress(R) \geq \stress(S_i) \geq 0$, and thus $R$ is crowded, so $S'_{\bullet}$ is a chain of crowded sets.

Thus, the toggle operator is an involution on  $\Chains(\StressedSets(M))$, which preserves the crowd hull, and which switches the length of the chain by $\pm 1$ whenever there is a set which is not a crowding record in the chain. So all terms coming from chains with sets that are not crowding records cancel, and we are reduced to summing over chains of crowding records.
\end{proof}

We now prove a version of Proposition~\ref{RecordStressorOnly} for flats, instead of sets:

\begin{proposition}\label{RecordStressorFlatsOnly}
In the second sum of Proposition~\ref{stressedChains}, we can instead sum only over chains of flats all of whose elements are crowding records.
\end{proposition}

\begin{proof}
The proof is very similar to that of Proposition~\ref{RecordStressorOnly}; we only record the differences.
We will denote our chains by $F_{\bullet}$ instead of $S_{\bullet}$ as a reminder that the elements are now assumed to be flats.

Let $F$ be a flat which is not a crowding record, and let $T$ be an overcrowded subset of $F$. Then $\cl_M(T) \subseteq F$ and $\stress(\cl_M(T)) \geq \stress(T)$. We claim that we can therefore always take $T_F$ in the previous proof  to be a flat. There is a small subtlety: We have to rule out the scenario that $\stress(T) = \stress(\cl_M(T)) =  \stress(F)$, with $\cl_M(T)$  a summand of $F$, although $T$ is not such a summand. But the only way that $\stress(T) = \stress(\cl_M(T))$  is if $T = \cl_M(T)$, so this does not occur. 
Thus, we take the set $T_F$ to be a flat whenever $F$ is a flat.

Let $F_{\bullet}$ be a chain of flats, one of whose flats is not a crowding record. As before, let $m$ be the smallest index such that $F_m$ is not a crowding record, and write $T:=T_{F_m}$, $U:=F_m \setminus T$. We once again write down a string of symbols $(\alpha_1, \alpha_2, \ldots, \alpha_m)$ in the alphabet $\{ \tt, \tu, \tv \}$, which this time are defined as follows: Let $M_i$ be the matroid $M|{F_i}/F_{i-1}$, a matroid on the ground set $F_i \setminus F_{i-1}$. We define $\alpha_i$ to be:
\begin{enumerate}
\item[$\tt$] if $T \cap (F_i \setminus F_{i-1})$ spans $M_i$.
\item[$\tu$] if $T \cap (F_i \setminus F_{i-1}) = \emptyset$.
\item[$\tv$] if $T \cap (F_i \setminus F_{i-1})$ is nonempty, but does not span $M_i$.
\end{enumerate}
We note that condition $\tu$ can equivalently be stated as $F_i \setminus F_{i-1} \subseteq U$. 

Once again, since $U$ is uncrowded, it does not occur in $F_{\bullet}$, so $(\alpha_1, \alpha_2, \ldots, \alpha_m)$ must contain either a $\tt \tu$ or a $\tv$; let $i$ be the first such position, as before.
If $(\alpha_i, \alpha_{i+1}) = (\tt,\tu)$, then we define a new chain of flats by deleting $F_i$. Conversely, if $\alpha_i = \tv$, then we define a new chain of flats by inserting $\cl_M(F_{i-1} \cup T \cap F_i)$. As before, these operations are mutually inverse. (Note that $\cl_M(F_{i-1} \cup T \cap F_i)$ is the flat of $M$ whose image in $M_i$ is the span of  $T \cap (F_i \setminus F_{i-1})$.) We now must repeat the previous computations of crowding.

As before, assume that we are in the case where $\alpha_i = \tv$, so the toggling operation is to insert $R:=\cl_M(F_{i-1} \cup T \cap F_i)$ into $F_{\bullet}$. Define $j$ as before. This time, we don't know that $F_{i-1} \cup T \cap F_i= F_j \cup T \cap F_i$, but we know that these sets have the same closure, since $T \cap (F_k \setminus F_{k-1})$ spans $M_k$ for $j < k \leq i-1$. So $R = \cl_M(F_j \cup T \cap F_i)$. Now, as before, $F_j \cap T = \emptyset$, and we deduce that $\stress(F_i) \leq \stress(F_j \cup T \cap F_i)$ as before. We then have $\stress(F_j \cup T \cap F_i) \leq \stress(R)$, so $\stress(F_i) \leq \stress(R)$, and we are done as before. 
\end{proof}

Thus, we can restrict our sums to running over chains of sets or flats in which every element of the chain is a crowding record.
Note that, in such a chain, we automatically have $\stress(S_0) \leq \stress(S_1) \leq \stress(S_2) \leq \cdots \leq \stress(S_{\ell})$. 
In fact, in  Theorem~\ref{FinalCancellationResult}  we will  restrict even further, to a sum where the inequalities are strict. 

We first introduce some notation and prove a lemma: Let $S \subseteq E$ be a crowding record. Let $S^1$, $S^2$, \dots, $S^r$ be the connected components of the matroid $M|S$, so $S = \bigsqcup S^i$. 
By Lemma~\ref{RecordSummand}, we have $\stress(S^i) \geq 0$ for all $i$.  We define $Y(S) = \bigcup_{\stress(S^i) > 0} S^i$ and $Z(S):=\bigcup_{\stress(S^i) = 0} S^i$.
\begin{lemma} \label{ZSummand}
Let $H' \subset H$ be crowding records.  Then $H' \cap Z(H)$ is a summand of both $Z(H)$ and  $Z(H')$.
\end{lemma}

\begin{proof}
First note that, since $Z(H)$ is a summand of $H$, the intersection $Z(H) \cap H'$ must be a summand of $H'$. 

Let $C$ be any connected component of $Z(H)$; we want to show that $C \cap H'$ is either $\emptyset$ or $C$. Since $C$ is a summand of $Z(H)$, which is a summand of $H$, we know that $C \cap H'$ is a summand of $H'$. Since $H'$ is a crowding record, we deduce that $\stress(C') \geq 0$. But also, $C$ is a summand of the crowding record $H$, so by Lemma~\ref{RecordSummand}  $C$ is a crowding record, and since  $C' \subseteq C$, so we deduce that $\stress(C') \leq 0$. Thus, $\stress(C') = 0$ and, using that $C$ is a crowding record, we deduce that $C'$ is a summand of $C$. Since $C$ is connected, this means that either $C' = \emptyset$ or $C'=C$, as desired.

We noted in the first paragraph that $Z(H) \cap H'$ must be a summand of $H'$. We have now shown that $Z(H) \cap H'$ is a direct sum of matroids of crowding zero, so we deduce more strongly that $Z(H) \cap H'$ is a summand of $Z(H')$, as desired.
\end{proof}

For the next theorem, we switch our conventions on chains so that they need not contain the empty set. For a chain  $\emptyset \subseteq H_0 \subset H_1 \subset H_2 \subset \cdots \subset H_m=E$, we  adjust our definition of 
$\verts(H_{\bullet}, a_M(H_{\bullet}))$ so that it consists of $p_M(H_i)$ for all $H_i \neq \emptyset, E$.  
With this notation, we will prove: 

\begin{theorem} \label{FinalCancellationResult}
We have 
\[  \omega^{\circ}(M) = \sum_{H_{\bullet}} (-1)^{c(H_0) + m-1} \cdot  \#\{ \textup{Ferroni paths weakly above $\verts(H_{\bullet}, a_M(H_{\bullet}))$} \} \]
where $c(H_0)$ is the number of connected components of $M|{H_0}$ and the sum is over chains of sets $\emptyset \subseteq H_0 \subset H_1 \subset H_2 \subset \cdots \subset H_m=E$ obeying the following conditions: 
\begin{enumerate}
\item All the $H_a$ are crowding records.
\item $0= \stress(H_0) < \stress(H_1) < \cdots < \stress(H_m)$.
\item We have $Z(H_0) \supseteq Z(H_1) \supseteq \cdots \supseteq Z(H_m)$.
\end{enumerate}
Moreover, the same is true if we impose additionally that the $H_i$ are flats.
\end{theorem}

 $\verts(H_{\bullet}, a_M(H_{\bullet}))$

\begin{proof}
We discuss the case of chains of sets; the result for flats is virtually identical.

We start with the sum in Proposition~\ref{RecordStressorOnly}, which runs over chains of crowding records, and we combine terms whose indexing chains have the same crowd hull. 
Thus, we first discuss the following question: Let $S_{\bullet}$ be a chain of crowding records (including $\emptyset$ and $E$, as usual) and let $H_{\bullet} = \StressHull(S_{\bullet})$. What can we say about $H_{\bullet}$?

Firstly, the last element of $S_{\bullet}$ is always in $\StressHull(S_{\bullet})$, so we have $H_m = E$. 
Secondly, since $H_{\bullet}$ is a subchain of $S_{\bullet}$, each element of $H_{\bullet}$ must be a crowding record. Thirdly, by the definition of the crowd hull, we have $\stress(H_0) < \stress(H_1) < \cdots < \stress(H_m)$. Finally, since all the $S_i$ are crowded, and since $\stress(S_0) = \stress(\emptyset) = 0$, we must have $\stress(H_0) = 0$. However, $H_0$ need not be $\emptyset$; rather, $H_0 = S_i$ where $i$ is the largest index such that $\stress(S_i) = 0$.

Thus, with that as motivation, we consider a chain of  crowding records  $H_0 \subset H_1 \subset H_2 \subset \cdots \subset H_m=E$  obeying  $0= \stress(H_0) < \stress(H_1) < \cdots < \stress(H_m)$, and we set out to understand all chains $S_{\bullet}$ of crowding records with $\StressHull(S_{\bullet}) = H_{\bullet}$.

Let $S_{\bullet}$ be such a chain. We renumber the elements of $S_{\bullet}$ as
\[\begin{array}{r@{\,}c@{\,}c@{\,}c@{\,}c@{\,}c@{\,}c@{\,}c@{\,}c@{\,}c@{\,}l}
&S_{00} &\subset& S_{01} &\subset& S_{02} &\subset& \cdots &\subset& S_{0 p_0} &= H_0  \\
\subset& S_{10} &\subset& S_{11} &\subset& S_{12} &\subset& \cdots &\subset& S_{1 p_1} &= H_1 \subset   \\
&&&&&&& \cdots && &   \\
\subset& S_{m0} &\subset& S_{m1} &\subset& S_{m2} &\subset& \cdots &\subset& S_{m p_m} &= H_m =E . \\
\end{array}\]
Using that all the sets $S_{ab}$ are all crowding records, but that only the $S_{a p_a}$ are in the crowd hull, we have
\[ \begin{array}{r@{\,}c@{\,}c@{\,}c@{\,}c@{\,}c@{\,}c@{\,}l}
\stress(S_{00}) &=&  \stress(S_{01}) &=& \stress (S_{02}) &=& \cdots &= \stress(S_{0 p_0}) \\
< \stress(S_{10}) &=&  \stress(S_{11}) &=& \stress (S_{12}) &=& \cdots &= \stress(S_{1 p_1}) \\
&&&&&& \cdots & \\
< \stress(S_{m0}) &=&  \stress(S_{m1}) &=& \stress (S_{m2}) &=& \cdots &= \stress(S_{m p_m}) .\\
\end{array} \]
Since $S_{a p_a}$ is a crowding record, each $S_{ab}$ must be a summand of $S_{a p_a}$ and, since they have the same crowding, $S_{ab}$ must be of the form $Z_{ab} \sqcup Y(H_a)$, where $Z_{ab}$ is a summand of $Z(H_a)$. 

Conversely, suppose that,  inside each $Z(H_a)$, we choose a chain of summands $Z_{a0} \subsetneq Z_{a1} \subsetneq \cdots \subsetneq Z_{a p_a}  \subseteq  Z(H_a)$, and we set $S_{ab} = Z_{ab} \cup Y(H_a)$. When will $S_{ab}$  be a chain of crowding records, containing $\emptyset$ and $E$, with crowd hull $H_{\bullet}$? We need three things:
\begin{enumerate}\renewcommand{\theenumi}{\roman{enumi}}
\item In order to get $S_{00} = \emptyset$, we must take $Z_{00} = \emptyset$.
\item In order to get $S_{ap_a} = H_a$, we must take $Z_{ap_a} = Z(H_a)$.
\item For $a \geq 1$, in order to get $H_{a-1} \subseteq S_{a0}$, we must take $Z_{a0}$ such that $H_{a-1} \subseteq Z_{a0} \cup Y(H_a)$. Since $H_{a-1} \subset H_a =  Z(H_a) \sqcup Y(H_a)$, it is equivalent to ask that $H_{a-1} \cap Z(H_a) \subseteq Z_{a0}$. 
\end{enumerate}
In short, we must have a chain of summands $\emptyset = Z_{00} \subsetneq Z_{01} \subsetneq \cdots \subsetneq Z_{0p_0} = H_0$ and, for each $a \geq 1$, we must have a chain of summands $H_{a-1} \cap Z(H_a) \subseteq Z_{a0} \subsetneq Z_{a1} \subsetneq \cdots \subsetneq Z_{ap_a} = Z(H_a)$. 
Note that the choices for different $a$ are independent, so we can count them separately and multiply.

We start with the $a=0$ factor. Let $H_0 = \bigsqcup_{i=1}^{c(H_0)} C_i$ be the decomposition of $H_0$ into connected components, so each $Z_{0b}$ must be of the form $\bigsqcup_{i \in I_b} C_i$. Also, $I_0$ must be $\emptyset$ and $I_{p_0}$ must be $[c(H_0)]$. So we are summing over chains in a boolean lattice of rank $c(H_0)$, with the requirement that we include the top and bottom element. This sum is $(-1)^{c(H_0)}$. 

Now, let $a \geq 1$.
By Lemma~\ref{ZSummand}, $H_{a-1} \cap Z(H_a)$ is a summand of $Z(H_a)$, so we can write $Z(H_a) = \left( H_{a-1} \cap Z(H_a) \right) \sqcup \bigsqcup_{i=1}^{\ell} C_i$ for various connected components $C_i$. So each $Z_{ab}$ is of the form $\left( H_{a-1} \cap Z(H_a) \right) \sqcup \bigsqcup_{i \in J_b} C_i$ for some index set $J_b$, and $J_{p_a}$ must be $\ell$, but $J_0$ is not required to be $\emptyset$. This time, we are summing over chains in a boolean lattice of rank $\ell$, but we are only required to include the top element, not the bottom. This sum is $-1$ if $\ell=0$ and is $0$ otherwise. So we only get a contribution from terms where $Z(H_a) =  H_{a-1} \cap Z(H_a)$, and that contribution is $-1$.

The condition $Z(H_a) =  H_{a-1} \cap Z(H_a)$ can be restated more nicely as $Z(H_a) \subseteq H_{a-1}$. 
But, from Lemma~\ref{ZSummand}, we already know that $H_{a-1} \cap Z(H_a)$ is a summand of $Z(H_{a-1})$, so we can restate this condition further as $Z(H_a) \subseteq Z(H_{a-1})$, as we stated it in the statement of the theorem.
\end{proof}

With Theorem~\ref{FinalCancellationResult}, we have achieved as much cancellation as is possible between chains with the same crowd hull; the distinct summands in Theorem~\ref{FinalCancellationResult} have distinct crowd hulls.
In the version of  Theorem~\ref{FinalCancellationResult} with chains of sets, we can still have two summands with the same Ferroni paths above them, because of the issue discussed in Remark~\ref{BetterHullRemark}. 
However, in the version with chains of flats, this cannot occur -- distinct summands in the version of Theorem~\ref{FinalCancellationResult} with flats always lie below different sets of Ferroni paths.

 Theorem~\ref{FinalCancellationResult} is extremely powerful in the version with flats, because the length of our chains is bounded both by the rank $r$ of $E$ and by the crowding $n-2r$ of $E$. Thus, we can get simple computations when either $r$ or $n-2r$ is small. We turn to this idea in the next two sections.

\addtocontents{toc}{\smallskip \textbf{Part 3: Special cases}}

\section{Matroids of low rank} \label{LowRankSection}

In this section, we will give formulas for $\omega(M)$ when $\rank(M) \leq 4$. We first note preliminary reductions: By Proposition~\ref{multiplicative}, we may as well assume that $M$ is connected,
which, further assuming $M$ is not a singleton loop, implies that $M$ is loop-free.
By \cite[Corollary 7.5]{KTutte} we can assume that $M$ does not have any pair of parallel elements,
i.e., as matroid theorists say, that $M$ is \newword{simple}.
Note that this last assumption means that all flats of rank one are singletons, and are therefore uncrowded.
\textbf{We make all of these assumptions for the remainder of Section~\ref{LowRankSection}.}

Under these reductions, if $\rank(M)=1$ 
then $M$ is a singleton coloop with $\omega(M)=0$;
it is however a standing assumption in \cite{KTutte} that matroids have no loops or coloops,
and collapsing parallel elements without creating coloops only reduces $M$ to $\Uniform{1}{2}$, with $\omega(\Uniform{1}{n})=\omega(\Uniform{1}{2})=1$ for $n\ge2$.
If $\rank(M)=2$, then our reductions imply $M$ is the uniform matroid $\Uniform{2}{n}$ with $n\ge3$, which has $\omega$ invariant $n-3$.

We move on to rank~3.

\begin{prop} \label{Rank3Formula}
Let $M$ be a rank~$3$ matroid on~$E$ satisfying our assumptions.
Let $L_1$, $L_2$, \dots, $L_{m}$ be the flats of rank~$2$. Then
\[ \omega(M) = \binom{|E|-4}{2} - \sum_i \binom{|L_i|-2}{2} . \]
\end{prop}

\begin{proof}
We put $n= |E|$ and $\ell_i = |L_i|$,
and apply Proposition~\ref{stressedChains} for flats.
Since the only crowded flats are the flats of rank $2$, the only chains of crowded sets are $(\emptyset, E)$ and those of the form $(\emptyset, L_i, E)$. 
The first chain of flats does not exclude any of the $\binom{n-4}{2}$ Ferroni paths from $(\frac12,\frac12)$ to $(n-3\frac12,2\frac12)$, 
each of which has $2$ diagonal steps and $n-6$ horizontal steps.
The chain of flats  $(\emptyset, L_i, E)$ requires the Ferroni paths to pass weakly over the point $(\ell_i-2,2)$. Such a path must have both its diagonal steps in within the first $\ell_i-2$ positions, so there are $\binom{|L_i|-2}{2}$ such paths. 
In the proposition the sum includes also the uncrowded rank~2 flats,
but these contribute $\binom{2-2}2=0$ or $\binom{3-2}2=0$.
\end{proof}

\begin{cor}
If $M$ is a matroid of rank $3$, then $\omega(M) \geq 0$. 

Among loop-free $M$, equality occurs if and only if
the simplification of~$M$ is either isomorphic to $\Uniform 35$ or
has a rank~$2$ flat containing all but two of its points.
\end{cor}

Note that a loop-free simple matroid $M$ of rank~$3$ that is also coloop-free is connected,
and if $M$ has a coloop then its complement is a rank~$2$ flat containing all but two (indeed all but one) of its points.

\begin{proof}
First, note that all of our previous reductions preserve the positivity of $\omega$, so we continue to make them. We also continue the notations $n$, $L_i$ and $\ell_i$ above.

We first claim that, in light of our previous reductions, we can find four elements in $E$ no three of which are dependent. Indeed, let $x_1$, $x_2$ and $x_3$ be a basis of $M$,  let $V_{ij}$ be the span of $\{ x_i, x_j \}$, and let $W_{ij} = V_{ij} \setminus \{ x_i, x_j \}$. If there is an element $z$ which is in none of the $V_{ij}$, then $\{ x_1, x_2, x_3, z \}$ has the desired property. Also, if two or more of the $W_{ij}$ are nonempty, then we are also done. Indeed, if $y_{ij} \in W_{ij}$ and $y_{jk} \in W_{jk}$ (with $i \neq k$) then $\{ x_i, x_k, y_{ij}, y_{jk} \}$ have the desired property. So the only case in which we are not done is if there is some $V_{ij}$, say $V_{23}$, such that every point other than $x_1$ lies on $V_{23}$. 
But then $M$ is disconnected, a contradiction.

So, let $p_1$, $p_2$, $p_3$, $p_4$ be four elements of $M$ no three of which are dependent. Put $P = \{ p_1,p_2, p_3, p_4 \}$ and $Q = E \setminus P$. Note that $\binom{n-4}{2}$ is the number of unordered pairs of distinct elements in $Q$. Each such pair spans a unique rank $2$ flat of $M$, and a rank $2$ flat $L_i$ is spanned by $\binom{|L_i \setminus P|}{2}$ many such pairs. So
\[\binom{n-4}{2} = \sum_i \binom{\ell_i - |L_i \cap P|}{2} . \]
By the condition on $ \{ p_1,p_2, p_3, p_4 \}$, we have $|L_i \cap P| \leq 2$ for all $i$, so 
\[ \sum_i \binom{\ell_i - |L_i \cap P|}{2}  \geq \sum_i \binom{\ell_i - 2}{2} . \]
Thus, $\omega(M) =  \binom{n-4}{2}  - \sum_i \binom{\ell_i - 2}{2} \geq 0$, as desired.

Equality occurs exactly when every pair of elements of~$Q$
has closure meeting $P$ in two points, 
i.e.\ lies in one of the flats $\cl_M(\{p_i,p_j\})$.
Call the collection of these six flats $\mathscr P$.
If $|Q|\ge2$ then we show all pairs from $Q$ determine the same flat in~$\mathscr P$.
Otherwise there exist $q_1,q_2,q_3\in Q$ so that the rank~$2$ flats $\cl(\{q_1,q_2\})$ and $\cl(\{q_1,q_3\})$ are unequal.
Then $\cl(\{q_2,q_3\})$ is unequal to either of these (if say $\cl(\{q_1,q_3\})=\cl(\{q_2,q_3\})=L$ then $q_1,q_2\in L$ so $\cl(\{q_1,q_2\})=L$).
Of these three distinct flats in $\mathscr P$, two must have intersection containing an element of~$P$, therefore of rank~$1$.
But this intersection also contains an element of $\{q_1,q_2,q_3\}$,
contradicting simplicity.
So $M$ has a single flat in~$\mathscr P$ containing all of~$Q$ plus two of the four elements of $P=E\setminus Q$.
If instead $|Q|<2$ there is no condition; 
the only case we haven't captured above is that in which 
$Q$ has exactly one element and it lies outside all the flats in $\mathscr P$,
which is the case $M\cong\Uniform 35$.
\end{proof}

In the same mode as Proposition~\ref{Rank3Formula} we can write down formulae for any fixed rank,
although the number of subsums grows exponentially.
Here is the result for rank~$4$.

\begin{prop} \label{Rank4Formula}
Let $M$ be a rank~$4$ matroid on~$E$ satisfying our assumptions.
Let $L_1$, $L_2$, \dots\ be the flats of rank~$2$, and $P_1$, $P_2$, \dots\ those of rank~$3$. 
Then
\begin{multline*}
\omega(M) = \binom{|E|-5}{3} 
- \sum_j \binom{|P_j|-3}{3} \\
- \sum_i \binom{|L_i|-2}{2}(|E|-\tfrac{2}{3}|L_i|-\tfrac{13}3)
+ \sum_{L_i\subset P_j} \binom{|L_i|-2}{2}(|P_j|-\tfrac{2}{3}|L_i|-\tfrac{7}{3}) . 
\end{multline*}
\end{prop}

\begin{proof}
As in the proof of Proposition~\ref{Rank3Formula}, we apply Proposition~\ref{stressedChains} for flats.
The four terms in the proposition arise from chains of crowded flats of the shapes 
$(\emptyset,E)$, $(\emptyset,P_j,E)$, $(\emptyset,L_i,E)$, and $(\emptyset,L_i,P_j,E)$ respectively.
As before, restricting these sums to run only over crowded flats only discards summands equal to~0.

For each shape of chain we count certain Ferroni paths.
We will explain this count just for the chains $(\emptyset,L_i,P_j,E)$,
as the other arguments are simplifications of this one.
For a Ferroni path from $(\frac12,\frac12)$ to $(|E|-4\frac12,3\frac12)$
to pass weakly over $p(L_i)=(|L_i|-2,2)$ and $p(P_j)=(|P_j|-3,3)$,
its first $|L_i|-2$ steps must include two or three of the three diagonal steps,
and all three diagonal steps must be among the first $|P_j|-3$.
When exactly two of the first $|L_i|-2$ steps are diagonal,
so the third falls strictly after index $|L_i|-2$ and not later than index $|P_j|-3$,
the number of such paths is $\binom{|L_i|-2}2\cdot((|P_j|-3)-(|L_i|-2))$.
The number of paths with all three diagonals among the first $|L_i|-2$ steps is
$\binom{|L_i|-2}3 = \binom{|L_i|-2}2\cdot\frac13(|L_i|-4)$.
The sum of these two counts is the fourth summand in the proposition.
\end{proof}

\section{Matroids with $n$ close to $2r$} \label{NearMiddleMatroidSection}

\subsection{The $\omega$ invariant for $n=2r$}\label{NearMiddleMatroidSection1}
 
\begin{proposition} \label{nIs2r}
Let $M$ be a rank $r$ matroid on~$E$ with $|E|=2r$. Then $\omega(M)$ is either $1$ or $0$. More specifically, let $M_1$, $M_2$, \dots, $M_c$ be the  connected components of $M$. Then $\omega(M)$ is $1$ if all of $M_i$ satisfy $|M_i| = 2 \rank(M_i)$ and have no proper nonempty crowded subsets, and is $0$ otherwise.
\end{proposition}

\begin{proof}
By Proposition~\ref{multiplicative}, $\omega(M) = \prod \omega(M_i)$, so we are reduced to computing $\omega(M_i)$.

Let $n_i = |E_i|$ and let $r_i = \rank_M(E_i)$. Since $n = \sum n_i$ and $r = \sum r_i$, either we have $n_i = 2 r_i$ for all $i$, or else there is some $i$ with $n_i < 2 r_i$. In that latter case, $\omega(M_i)=0$ so $\omega(M)$ vanishes and we are done. Thus, we may (and do) assume that $n_i = 2 r_i$ for all $i$.

First, suppose that $M_i$ has a proper nonempty crowded subset $S$. Then $S$ is overcrowded in $M_i$, so $\omega(M_i)=0$ by Corollary~\ref{overStressed}.

On the other hand, suppose that none of the $M_i$ have any nonempty proper crowded subsets. Then Proposition~\ref{NoStress} shows that $\omega(M_i) = \binom{n_i-r_i-1}{r_i-1} = 1$ for all $i$, so $\omega(M) = 1$.
\end{proof}

\subsection{The $\omega$ invariant for $n=2r+1$}\label{NearMiddleMatroidSection2}
For this subsection we consider a rank $r$ matroid $M$ on~$E$ with $|E|=2r+1$.
We begin with preliminary reductions:

\begin{prop}
 Let $M_1$, $M_2$, \dots, $M_c$ be the  connected components of $M$, with $M_i$ of cardinality $n_i$ and rank $r_i$.
 If we have $n_i < 2 r_i$ for any index $i$, then $\omega(M)=0$. Otherwise, we can reorder the components such that $n_1 = 2 r_1 + 1$ and $n_i = 2 r_i$ for $2 \leq i \leq c$. 
 If any of $M_2$, $M_3$, \dots, $M_c$ has a  nonempty proper crowded subset, then we again have $\omega(M)=0$. 
 Finally, if $n_1 = 2 r_1 + 1$ and $n_i = 2 r_i$ for $2 \leq i \leq c$, and furthermore none of $M_2$, $M_3$, \dots, $M_c$ has a  nonempty proper crowded subset, then $\omega(M) = \omega(M_1)$.
  \end{prop}
  
  \begin{proof}
  We have $\omega(M) = \prod \omega(M_i)$. If $n_i < 2 r_i$ for some $i$, then $\omega(M_i)=0$ so $\omega(M)=0$.
   
  If we have $n_i \geq 2 r_i$ for all $i$, then using that $n = \sum n_i$, $r = \sum r_i$ and $n=2r+1$, we see that we have $n_j = 2 r_j+1$ for one index $j$, and $n_i = 2 r_i$ for all the other indices $i$. Without loss of generality, let $M_1$ be the component where $n_1 = 2 r_1 + 1$.
  
  From Proposition~\ref{nIs2r}, if $M_i$ has a nonempty proper crowded subset for $2 \leq i \leq c$, then $\omega(M_i)=0$ so $\omega(M)=0$. 
If none of $M_2$, $M_3$, \dots, $M_c$ has a  nonempty proper crowded subset, then $\omega(M_2) = \omega(M_3) = \cdots = \omega(M_c) =1$ by Proposition~\ref{nIs2r}, so $\omega(M) = \omega(M_1)$. 
  \end{proof}

Thus, from now on, we rename $M_1$ to $M$, which is to say \textbf{we assume from now on that $M$ is connected}. 
Moreover, if $M$ has a proper subset of positive crowding, Corollary~\ref{overStressed} implies $\omega(M)=0$.
So \textbf{we assume from now on that $\stress(S) \leq 0$ for every proper subset $S$ of $M$}.
It will also be convenient to \textbf{assume from now on that $r > 0$}; the only matroid which this excludes is the uniform matroid of rank $0$ on a singleton.

Let $T_1$, \dots, $T_p$ be the minimal nonempty crowded sets of~$M$.  We will need some lemmas about these before proceeding to our final computation:
\begin{lemma} \label{pgeq2}
There are at least two  minimal nonempty crowded sets of~$M$.
\end{lemma}

\begin{proof}
We first show that there is at least one such set. Since $r > 0$, we have $|E| \geq 3$ and so, for any $i \in E$, the set $E \setminus \{ i \}$ is nonempty. Then $E \setminus \{ i \}$ is a nonempty crowded subset, so $E \setminus \{ i \}$ must contain a minimal nonempty crowded subset, giving us one set $T_1$.

Now, let $j \in T_1$. Then, as above, $E \setminus \{ j \}$ must contain a minimal nonempty crowded subset, and this subset is not $T_1$, so there is at least one more.
\end{proof}

\begin{lemma}\label{2r+1Dichotomy}
Let $M$ and $T_1$, \dots, $T_p$ be as above. Then either:
\begin{enumerate}
\item[Case 1:] For every $i \neq j$, we have $T_i \cap T_j = \emptyset$ or else
\item[Case 2:] For every $i \neq j$, we have $T_i \cup T_j = E$ .
\end{enumerate}
\end{lemma}

\begin{proof}
We first begin by showing that, for an individual pair of indices $(i,j)$, we either have  $T_i \cup T_j = E$ or $T_i \cap T_j = \emptyset$.
Suppose, to the contrary, that neither is true.
By the supermodularity of crowding, 
\[ \stress(T_i \cup T_j) + \stress(T_i \cap T_j) \geq \stress(T_i) + \stress(T_j) = 0 . \]
But $T_i \cup T_j$, by assumption, is not $E$, so $\stress(T_i \cup T_j) \leq 0$.
And $T_i \cap T_j$, by assumption, is nonempty, and is a proper subset of the minimal crowded set $T_i$, so $\stress(T_i \cap T_j) < 0$. 
This is a contradiction, and we have established the claim for an individual $(i,j)$. 

We also note that it is impossible to have both $T_i \cap T_j = \emptyset$ and $T_i \cup T_j = E$ for some $(i,j)$. Indeed, since $T_i$ and $T_j$ have crowding $0$, they have even cardinality, but $|E|$ is odd.

So, for every $(i,j)$, exactly one of the two conditions above holds. We now must show that it is the same choice for every pair $(i,j)$. Suppose, for the sake of contradiction, that we have three distinct indices $i$, $j$, $k$ with $T_i \cap T_j = T_j \cap T_k = \emptyset$ and $T_j \cup T_k = E$.
Then
\[ T_i \cap (T_j \cup T_k)  = T_i \cap E = T_i \ \textup{but} \ (T_i \cap T_j) \cup (T_i \cap T_k) = \emptyset \cup \emptyset , \]
contradicting the distributivity of $\cap$ over $\cup$. Similarly, if $T_i \cup T_j = T_j \cup T_k = E$ and $T_j \cap T_k = \emptyset$, we obtain a contradiction to the distributivity of $\cup$ over $\cap$.
\end{proof}

We are now ready to begin our computation of $\omega(M)$. 
We will use Theorem~\ref{FinalCancellationResult}, in the version where we sum over sets of crowding records. Our sum is over chains $H_0 \subset H_1 \subset \cdots \subset H_m = E$ where (among conditions), we have $\stress(H_0) = 0$ and $\stress(H_0) < \stress(H_1) < \cdots < \stress(H_m) = 1$. So our chains are all of length $1$, of the form $(H_0, E)$. We will shorten $H_0$ to $H$. In light of our previous reductions, the only condition on $H$ is that it is a crowding record of crowding $0$,  and the only condition on our Ferroni paths is that they pass weakly above $p_M(H)$.

\begin{lemma} \label{PathCount}
Let $H$ be a set of crowding $0$. The number of Ferroni paths passing weekly above $p_M(H)$ is $r - \rank(H)$.
\end{lemma}

\begin{proof}
Since $H$ has crowding $0$, we have $|H| = 2 \rank(H)$, so $p_M(H) = (\rank(H), \rank(H))$. A Feronni path lying weakly above this point must start with $\rank(H)$ diagonal steps. The remaining $r-1-\rank(H)$ diagonal steps and single horizontal step can be shuffled in any order, giving $r-\rank(H)$ choices. 
\end{proof}

So, our formula for $\omega(M)$ simplifies to
\begin{equation}
 \sum_H (-1)^{c(H)} (r-\rank(H)) \label{SumToDoStress1}
 \end{equation}
where the sum is over $H$ a crowding record of crowding $0$.

We will now compute~\eqref{SumToDoStress1} in each of the two cases of Lemma~\ref{2r+1Dichotomy}.

\begin{prop}
Suppose that we are in Case 1 of Lemma~\ref{2r+1Dichotomy}, meaning that the $T_i$ are disjoint. Then the crowding records of crowding zero are exactly the sets of the form $\bigsqcup_{i \in I} T_i$, and we have  $\omega(M) = 0$.
\end{prop}

\begin{proof}
We first note that $\bigcup_i T_i \neq E$. Indeed, the $T_i$ are disjoint and of even cardinality, and $|E|$ is odd. Similarly, the union of any subcollection of the $T_i$ is also not $E$.

We next claim that any union $\bigcup_{i \in I} T_i$ is of crowding $0$. This follows by induction on the cardinality of $I$. Suppose that we know the result for smaller index sets and choose some $j \in I$. Then we have 
\[ \stress(\bigcup_{i \in I} T_i) \geq   \stress(\bigcup_{i \in I \setminus j} T_i) + \stress(T_j) -  \stress(\emptyset) = 0+0-0 = 0 \]
where $ \stress(\bigcup_{i \in I \setminus j} T_i) =0$ by induction, and the intersection is empty by disjointness. So the union is crowded. But $E$ is the only set of positive crowding, so $\bigcup_{i \in I} T_i$ must be $\emptyset$, completing the induction.

We now claim that these unions are crowding records. Indeed, let $V \subseteq T_{i_1} \cup T_{i_1} \cup \cdots \cup T_{i_k}$ be a set of crowding $0$. Then, by minimality of the $T_i$'s and by the absence of sets of positive crowding,  we must have $\stress(V \cap T_{i_j}) = \emptyset$ or $T_{i_j}$ for each $T_{i_j}$. 

Finally, let $W$ be any crowding record, and let $T = \bigcup_{T_i \subseteq W} T_i$. By the above, $T$ is a set of crowding $0$ contained in $W$, so $T$ must be a summand of $W$. If $T \neq W$, then $T \setminus W$ is a set of crowding $0$, and this set must contain some $T_j$, contradicting the construction of $T$. So $T= W$, and we have shown that every crowding record is of the form $\bigsqcup_{i \in I} T_i$.

Having proved all of this, we compute $\omega(M)$.
Set $r_i= \rank_M(T_i)$, so $|T_i| = 2 r_i$.
So \eqref{SumToDoStress1} gives
\[ \omega(M) = \sum_{I \subseteq \{1,2,\ldots, p \}} (-1)^{|I|-1} \sum_{i \in I} (r-r_i) . \]
We can compute this as
\[ \left. \frac{d}{dx}  \right|_{x=1} x^r  \prod_{i = 1}^p (1-x^{-r_i}) . \]
Since $p \geq 2$ (Lemma~\ref{pgeq2}), this is $0$.
\end{proof}

We now consider Case 2:

\begin{prop}
Suppose that we are in Case 2 of Lemma~\ref{2r+1Dichotomy}, meaning that  $T_i \cup T_j = E$ for all $i \neq j$. Then the crowding records of crowding $0$ are $\emptyset$ and the $T_i$.  Moreover, $p$ is odd, and $\omega(M) = \tfrac{p-1}{2}$.
\end{prop}

Note that $p \geq 2$ (Lemma~\ref{pgeq2}) so, in this case where $p$ must be odd, we have $p \geq 3$ and  thus $\omega(M) = \tfrac{p-1}{2} \geq 1$. This is the only case where $\omega(M)$ is nonzero.

\begin{proof}
By their definition, the $T_i$ are crowding records. 

Let $V$ be any nonempty crowding record of crowding $0$, we must show that it is one of the $T_i$. By the minimality of the sets $T$, there is some $T_i \subseteq V$; suppose for the sake of contradiction that this $T_i$ is not all of $V$. Since $V$ is a crowding record, $T_i$ is a summand of $V$, and we deduce that $\stress(V \setminus T_i)=0$, so there is some $T_j \subseteq V \setminus T_i$. But then $T_i \cap T_j = \emptyset$, and we are in Case 1 rather than Case 2. We have now verified that the $T_i$ are all of the nonempty crowding records of crowding $0$. Also, each $T_i$ is connected, as otherwise its connected components would be smaller crowded sets.
So~\eqref{SumToDoStress1} reads
\[ \omega(M) =  r - \sum_i (r-\rank(T_i))  .\]
We now must simplify this sum. 

We set $U_i = E \setminus T_i$. So the $U_i$ are disjoint. As in the proof of  Lemma~\ref{pgeq2}, for any $j \in [n]$, there is some minimal nonempty crowded set $T_i$ contained in $[n] \setminus \{ j \}$, and thus $j$ is in $U_i$ for this $i$. So the $U_i$ partition $E$. 

Put $r_i = \rank_M(T_i)$, so $|T_i| = 2 r_i$ and $|U_i| = n - 2 r_i$. Since the $U_i$ partition $E$, we have $\sum (n-2 r_i) = n$ or, in other words, $pn - 2 \sum r_i = n$. In particular, since $n$ is odd, this shows that $p$ is odd, as claimed. We rewrite this one more time for future use: $\sum r_i = \tfrac{p-1}{2} n$.

We get
\[ \omega(M) = r - \sum_i (r-r_i) = r - pr + \sum r_i = r - pr + \tfrac{p-1}{2} n = \tfrac{p-1}{2} ( n - 2r) = \frac{p-1}{2} . \qedhere \]
\end{proof}

\bibliography{Omega}{} \bibliographystyle{alpha}

\end{document}